\documentclass[12pt]{article}
\usepackage{matlab-prettifier}
\usepackage{bm,amssymb,amsmath,graphicx,amsthm,wrapfig,caption,subcaption,epstopdf,esint,cancel,xcolor,soul,cancel,ulem,a4wide}

\newtheorem{lemma}{Lemma}
\newtheorem{example}{Example}

\newtheorem{theorem}{Theorem}
\newtheorem{proposition}{Proposition}
\newtheorem{conjecture}{Conjecture}
\newtheorem{problem}{Problem}

\newtheorem{assumption}{Assumption}

\theoremstyle{remark}

 \numberwithin{equation}{section}
 
 \newcommand{\Real}{\mathbb{R}}
  \newcommand{\Integer}{\mathbb{Z}}

\title{Optimal regulation in a periodic environment: insights from a simple model}
\author{Nir Gavish\footnote{Technion Israel Institute of Technology, Faculty of Mathematics, Technion City, Haifa 3200003, Israel, ({\tt ngavish@technion.ac.il})} 
\and Guy Katriel\footnote{Department of Applied Mathematics, Braude College of Engineering, Karmiel,  2161002, Israel, ({\tt katriel@braude.ac.il})}}
\date{}

\begin{document}
\maketitle
\begin{abstract}
We perform a detailed study of 
a simple mathematical model addressing the problem of optimally regulating a process subject to periodic external forcing, which is interesting both in view of its direct applications and as a prototype for more general problems. In this model one must determine an optimal time-periodic `effort' profile, and the natural setting for the problem is in a space of periodic non-negative measures. We prove that there exists a unique solution for the problem in the space of measures, and then turn to characterizing this solution.
Under some regularity conditions on the problem's data, we prove that its solution is an absolutely continuous measure, and provide an explicit formula for the measure's density. On the other hand, when the problem's data is discontinuous, the solution measure can also include atomic components, representing a concentrated effort made at specific time points. Complementing our analytical results, we carry out numerical computations to obtain solutions of the problem in various instances, which enable us to examine the interesting ways in which the solution's structure varies as the problem's data is varied.
\end{abstract}

\section{Introduction}

What is the optimal way to allocate effort and resources to essential activities when living in a fluctuating environment, in which costs and/or rewards for activities vary periodically in time?
Biological organisms are faced with this problem, as they are subject to daily (circadian) and seasonal environmental variation. 
Biological clocks, that is, internal periodic mechanisms that allow organisms to track time,
enable to predict regular environmental variation, and accordingly regulate activities and processes such as sleep-wake cycles, metabolism, body temperature, immune function, reproduction, foraging, and other behavioral patterns \cite{Forger2024biological,Kreitzman,Kumar,Perreau}.
However, even under the idealized assumption that
environmental variation is strictly periodic and known,
determining the optimal regulation policy is a
highly nontrivial problem.

In this work, we aim to enhance our understanding
of the question of optimal regulation in a periodic environment by performing a detailed study of a minimal model, which allows us to explore some of the mathematical patterns which emerge in such problems. The model, which consists of a linear scalar differential equation, lends itself to a variety of natural interpretations. Its simplicity allows us to make considerable progress in characterizing the optimal solutions. However, despite this simplicity, we will see that it gives rise to intricate and subtle phenomena and that
its detailed understanding requires
delicate mathematical analysis.

\subsection{Introducing the problem}
\label{sec:intrducing}

We present the problem of study with one illustrative example of an optimal clearing problem. It is important to mention that the mathematical problem is applicable to numerous scenarios, as discussed in more detail in the Discussion section.
Assume a pollutant is released into an environment at time-dependent rate 
$c(t)\geq 0$ (mass per unit time), which will be assumed to be given. The pollutant is removed or degraded as a first-order process with a base rate $\delta>0$ (per unit time), which can be increased by expending an additional time-dependent clearing effort $\eta(t)\geq 0$, which will be assumed to be under our control. The amount of pollutant $S(t)$ in the environment is then given by the linear differential equation
\begin{equation}\label{eq:s0}S'(t)=c(t) -[\eta(t)+\delta ]S(t).\end{equation}
We will assume that both the pollutant inflow rate $c(t)$ and the clearing effort $\eta(t)$ are 
$T$-periodic functions ($T>0$). It is a standard fact that \eqref{eq:s0} has a unique periodic solution
$S(t)$ (an explicit expression for which will be given below), and that all solutions $\tilde{S}(t)$ of
\eqref{eq:s0} will satisfy
$$\lim_{t\rightarrow \infty}[\tilde{S}(t)-S(t)]=0,$$
implying that the weighted time-averaged amount of pollutant is
\begin{equation}\label{eq:av}\bar{S}=\lim_{t\rightarrow \infty}\frac{1}{t}\int_0^t w(t) \tilde{S}(\tau)d\tau=\frac{1}{T}\int_0^T w(t)S(t)dt,
\end{equation}
where the weight $w(t)$ is a $T$-periodic function
reflecting the cost of pollution, which may vary
according to season (as a particular case one can take $w(t)$ to be constant, as we shall do in our numerical examples).

Our aim will be to choose the effort profile
$\eta(t)$  so as to minimize the cost $\bar{S}$, under the 
constraint of a given time-averaged effort 
$\bar{\eta}$:
\begin{equation}\label{eq:ce}\frac{1}{T}\int_0^T \eta(t)dt=\bar{\eta}.\end{equation}
The optimal choice of the effort profile
 obviously depends on the inflow rate $c(t)$,
as well as on the weight $w(t)$ and the value of the
total effort constraint $\bar{\eta}$, and we aim to understand this dependence. 

To appreciate the intricacy of the optimization problem involved, note that, at any point in time,
the efficacy of a unit effort employed in clearing the pollution depends linearly on the current amount of pollutant in the environment. Thus, it seems
advisable to concentrate effort at times in which
the level of pollutant is high. However, this level itself depends on both the past inflow rate of the pollutant
(given by $c(t)$) and on all previous clearing efforts. 
In view of this circularity, it is far from obvious how to best spread efforts over time. In particular, two qualitative issues are of interest:
(i) Are there
certain time periods in which it is best not to
make any effort (i.e., set $\eta(t)=0$)? (ii) Are there times
when it is best to make a highly
concentrated effort (expressed mathematically as a $\delta$-function)? These questions will be among those addressed by our results.

\subsection{Formalization of the problem}

In the above description, we have been vague
about the assumptions made concerning the given functions $c(t),w(t)$ and the class of effort profiles $\eta(t)$ over which we optimize. We will now be precise.

We first fix some notations. 
$L^p(T)$ ($1\leq p\leq \infty,T>0$) will denote the space of 
$T$-periodic functions whose restriction to $[0,T]$ is $L^p$. $C(T)$ we will denote the space of continuous $T$-periodic functions.
$L^p_+(T),C_+(T)$ will denote the set of non-negative functions in these spaces.


We will assume that $c\in L^1_+(T)$, and $c$ is not identical to $0$ (further restrictions on $c$ will be imposed later). 
We note that the standard existence theory for solutions of ordinary differential equations (\cite{hale2009ordinary}, sec. I.5) ensures the existence of an absolutely continuous solution, with a given initial value, of \eqref{eq:s0}, if $\eta\in L^1(T)$. Among all solutions there is a unique 
$T$-periodic solution, given explicitly by
\begin{equation}\label{eq:spex}
\begin{split}
	S(t)=&\frac{1}{1-e^{-\int_0^T[\eta(s)+\delta]ds}}\times\\&\left[\int_0^t  c(r)e^{-\int_r^t [\eta(s)+\delta]ds}dr+e^{-\int_0^T [\eta(s)+\delta]ds}\int_t^T c(r) e^{-\int_r^t[\eta(s)+\delta]ds}dr\right].
    \end{split}
\end{equation}

For our purposes, however, it will be important to allow a more general 
space of effort profiles $\eta(t)$: Instead of assuming $\eta(t)$ to be 
a function in $L^1_+(T)$, we will 
assume it to be a non-negative {\it{measure}}. This extension is essential for several reasons. On the one hand, it will allow us to prove an existence (and uniqueness) result for our optimization problem using 
a compactness argument, which is not
possible when working in the space $L^1$. Indeed, we shall see that under some circumstances the optimal solution contains atomic components ($\delta$ functions) which can only be accommodated by allowing our solutions to be measures.
Moreover, this extension is also natural from a modeling viewpoint, as the atomic components in an effort profile $\eta(t)$ represent a concentrated effort which (in terms of the example above) instantly removes a fraction of the pollution. 

In our generalized formulation of the
problem, we replace the function $\eta(t)$
by a Borel regular positive $T$-periodic measure $\mu$ on $\Real$, where, for 
$t_1<t_2$, 
$\mu([t_1,t_2))$
denotes the total effort expended during this time interval.
The periodicity means that
we have
$\mu([t_1+T,t_2+T))=\mu([t_1,t_2))$ for all $t_1<t_2$.

We recall that such measures are in one-to-one correspondence with the set ${\cal{M}}(T)$ of all right-semicontinuous monotone non-decreasing functions $\alpha(t)$ on $\Real$, normalized so that $\alpha(0)=0$, and satisfying
\begin{equation}\label{eq:aper}\alpha(t+T)=\alpha(t)+\alpha(T),\quad\forall t.\end{equation}
Each $\alpha\in {\cal{M}}(T)$ is associated with the $T$-periodic measure  defined by
$$\mu_{\alpha}([t_1,t_2))=\alpha(t_2)-\alpha(t_1).$$
Thus $\alpha(t)$ is the cumulative effort made over the time interval $[0,t)$, and will thus be referred to as the {\it{cumulative effort profile}}.

The space $L^1_+(T)$ is embedded in ${\cal{M}}(T)$, 
by associating with $\eta\in L^1_+(T)$ the function
\begin{equation}\label{eq:ea}\alpha(t)=\int_0^t \eta(s)ds,\end{equation}
and we then have
$$\alpha(t_2)-\alpha(t_1)=\mu_{\alpha}([t_1,t_2))=\int_{t_1}^{t_2}\eta(t)dt,$$
so that a function $\eta\in L^1[0,T]$
corresponds to an {\it{absolutely continuous}}
measure $\mu_{\alpha}$.

Each $\alpha\in {\cal{M}}(T)$ is associated in a one-to-one manner with the positive linear functional $\alpha^*\in C(T)^*$ (the dual space of $C(T)$)
 defined by the Stieltjes integral
$$f\in C(T)\quad\Rightarrow\quad
\alpha^*[f]=\int_0^T f(t)d\alpha(t).$$

With the identification
of $\eta\in L^1_+(T)$ with $\alpha\in {\cal{M}}(T)$ given by \eqref{eq:ea}, the constraint on the cumulative effort per period given by \eqref{eq:ce} generalizes to
$$\frac{1}{T}\cdot \alpha(T)=\bar{\eta}.$$
For $\bar{\eta}>0$, we denote the set of 
cumulative effort profiles satisfying this constraint by
$${\cal{M}}(T,\bar{\eta})=\{ \alpha\in {\cal{M}}(T),\quad\alpha(T)=T\bar{\eta}\}.$$

Expression~\eqref{eq:spex} for the periodic solution $S(t)$
generalizes to~$S[\alpha](t)$ by 
replacing~$\int_r^t \eta(s)ds$ with $\int_r^t d\alpha(s)=\alpha(t)-\alpha(r)$, giving, for any $\alpha\in {\cal{M}}(T,\bar{\eta})$,
 \begin{equation}\label{eq:spex0}
 	S(t)={\cal{S}}[\alpha](t)\doteq\frac{e^{-\alpha(t)-\delta t }}{1-e^{-[\bar{\eta}+\delta ]T}}\left[\int_0^t  c(r)e^{\alpha(r)+\delta r}dr+e^{-[\bar{\eta}+\delta ]T}\int_t^T c(r)e^{\alpha(r)+\delta r}dr\right],
 \end{equation}
 which we can write more compactly as
\begin{subequations}\label{eq:spex1}
\begin{equation}
	S(t)={\cal{S}}[\alpha](t)\doteq e^{-\alpha(t)-\delta t }\int_0^T K(t,r)c(r)e^{\alpha(r)+\delta r}dr,
\end{equation}
where
\begin{equation}
K(t,r)=\begin{cases}
\frac{1}{1-e^{-[\bar{\eta}+\delta ]T}}& r<t\\
\frac{1}{e^{[\bar{\eta}+\delta ]T}-1}&r\ge t
\end{cases}.
\end{equation}
\end{subequations}
When $\eta\in L^1_+(T)$ and $\alpha$ is defined by \eqref{eq:ea}, we have that \eqref{eq:spex1} coincides with \eqref{eq:spex}.
We note that (using monotonicity of $\alpha$) we have, for $t\in [0,T)$,
$$0\leq S(t)\leq 
\frac{1}{1-e^{-[\bar{\eta}+\delta T]}}e^{-\alpha(t)-\delta t }\left[e^{\alpha(t)+\delta t}\int_0^t  c(r)dr+\int_t^T c(r)dr\right]$$
$$=\frac{1}{1-e^{-[\bar{\eta}+\delta T]}}\left[\int_0^t  c(r)dr+e^{-\alpha(t)-\delta t }\int_t^T c(r)dr\right]\leq \frac{\|c\|_{L^1(T)}}{1-e^{-[\bar{\eta}+\delta T]}},$$
and this inequality extends to all $t$ by periodicity of $S(t)$,
so that \eqref{eq:spex1} defines an 
operator ${\cal{S}}:{\cal{M}}(T)\rightarrow L^{\infty}(T)$.

Given a cumulative effort profile
$\alpha\in {\cal{M}}(T,\bar{\eta})$, we define the cost
\begin{equation}\label{eq:dPhi}
\Phi[\alpha]=\int_0^T w(t){\cal{S}}[\alpha](t)dt
=
\int_0^T w(t) e^{-\alpha(t)-\delta t }\int_0^T K(t,r) c(r) e^{\alpha(r)+\delta r}drdt
\end{equation}
where the weight function $w(t)$ is in $L^1_+(T)$
is not identically $0$. 

We have therefore defined a functional $\Phi: {\cal{M}}\rightarrow [0,\infty)$,
and we can formulate our main problem:
\begin{problem}\label{prob:main}
	Given $\bar{\eta}>0$, find $\alpha\in {\cal{M}}(T,\bar{\eta})$ for which 
	$\Phi[\alpha]$ is minimal.
\end{problem}
The {\it{data}} of our problem are thus the inflow rate function $c(t)$, the
weight function $w(t)$, and the cumulative effort constraint $\bar{\eta}$. The solution of our problem will be denoted $\alpha_{opt}$, which is the {\it{optimal cumulative effort profile}}, and the corresponding 
measure $\mu_{opt}$ will be called the 
{\it{optimizing measure}}.
Our aim is to characterize this measure, which we shall do using both
analytical methods and numerical 
computations.

\subsection{Related studies}

Our problem is an optimal control problem \cite{Macki}, and as such can be related to several strands in the vast literature
on optimal control.

The field of optimal periodic control aims to optimize the performance of systems using a time-periodic control signal \cite{bettiol2024principles,colonius2006optimal,gilbert1977optimal}. Applications includes chronotherapy~\cite{adam2019emerging}, 
drug administration~\cite{cortes2022dosing,ghanaatpishe2017structure,ghanaatpishe2017online,varigonda2008optimal},
harvesting \cite{Maarten,Tang1,Xiao}, pest control \cite{cardoso2009multi,Tang2}, chemotherapy~\cite{belfo2021optimal,cacace2018optimal,ledzewicz2004structure,swierniak2010periodic}, and vaccination~\cite{bai2012global,doutor2016optimal,moneim2005use,Shulgin}.
The focus of the above-mentioned literature is split between two categories of problems.  The first category of problems are
autonomous when subject to a constant
control, so that the central question is whether
introducing periodicity through the control can
improve upon the performance obtained under the best constant control. The second category of problems are intrinsically periodic, e.g., as in our case due to the periodicity of the inflow rate $c(t)$.  Our study focuses on the second type, and addresses the central question of how to {\it{exploit}} the inherent 
periodicity in an optimal manner, by choosing the appropriate effort profile. 

In the recent work \cite{ali2021maximizing}, periodic 
control problems were considered 
both in the case of an autonomous and
of a non-autonomous system. In particular, 
the problem considered in Section 4.3 of
\cite{ali2021maximizing} is equivalent to
the problem considered here, in the case where $c(t)$ is constant
(by setting $S(t)=1-x_0(t),c(t)=\lambda_1,w(t)=\lambda_1\beta(t),\delta=\lambda_1,\eta(t)=u_0(t)$ in the formulation of \cite{ali2021maximizing}), but
with the control ($\eta(t)$ in our notation) being a function restricted to lie
between two given positive values, and the weight $w(t)$ everywhere differentiable.
The boundedness of the controls allows one to obtain existence of an optimizer working in space $L^\infty$ (using Fillipov's theorem) and
to apply the Pontryagin Maximum Principle, together with an analysis of singular trajectories, to obtain an explicit expression for the
optimal control, on the set in which it lies strictly between the two given positive values, an expression which is essentially the same as the one we obtain in Theorem \ref{th:main} (Section \ref{sec:char}). 
In the current work, however,
we expand the space of controls over which the optimization problem is posed to include measures (rather than bounded functions), an extension which, as we show, is essential
when one wants to allow the data $c(t),w(t)$ to include discontinuous functions. This precludes the use of results in
optimal control theory which assume bounded controls.

Many works investigate optimal control problems in which the control is restricted {\it{a priori}} to consist of discrete 
impulses ($\delta$-functions), with applications 
such as impulsive harvesting \cite{Tang1,Xiao,Maarten}, impulsive pest control \cite{cardoso2009multi,Tang2}, chemotherapy~\cite{belfo2021optimal,cacace2018optimal}, and pulse vaccination \cite{Shulgin}. Here, in contrast, we are interested in a problem in which the optimal
control may be a mixture of continuous and impulsive components, in order to understand the conditions under which discrete impulses emerge as components of the optimal control. 

A theory of optimal control problems
in which the control is a measure has been developed under the title of Optimal Impulsive Control Theory
\cite{arutyunov2019optimal}, including generalizations of the Pontryagin Maximum Principle to such settings. 
This theory needs to address many
technical issues, since even the notion of solution for a differential 
involving coefficients which are 
measures needs to be carefully defined.
In the current work we have taken 
a different approach, exploiting the
fact that the periodic solution to our differential equation can be written in an explicit integral form, as in
\eqref{eq:spex1}, so that our problem 
becomes one of minimizing the explicit functional $\Phi$, see \eqref{eq:dPhi}. Working with this functional, which turns out to be convex, has enabled us to derive a first order condition which is necessary and sufficient for a measure to be an optimizer. Analysis of the first order condition allows us to characterize the optimizer, including regularity results, the explicit expression for the absolutely continuous component of the optimizing measure, and results on the atomic
component of the measure which arises 
when the data is discontinuous.

%

\subsection{Overview}

In Section \ref{sec:existence} we prove that the problem \ref{prob:main} has a unique solution in the space of positive measures $\alpha\in {\cal{M}}(T,\bar{\eta})$. Along the way we will show that the functional
$\Phi$ is convex, a fact which will be crucial for our arguments, since it implies that the first-order condition for 
optimality,
which is derived in Subsection
\ref{sec:foc}, is both necessary and sufficient.

Our main theorems characterizing the solution of Problem \ref{prob:main}
are presented and discussed in Section 
\ref{sec:char}, together 
with results of numerical computations which illustrate and illuminate the analytical results. Proofs of these results are given in Section \ref{sec:proofs}.

Theorem \ref{th:main} shows that when the data $c,w$ satisfy some regularity assumptions, the solution of problem \ref{prob:main} is absolutely continuous and provides an explicit expression for this solution on its support, the range of time at which effort is expended. When the cumulative effort allowed is
sufficiently large, that is when $\bar{\eta}$ is larger than a critical value, we show in Theorem \ref{th:pos} that 
the support of $\mu_{opt}$ is the entire real line, that is effort will be expended at all times. On the other hand, in
Theorem \ref{prop:eta0} we show that when $\bar{\eta}$ 
goes to $0$, the support of $\mu_{opt}$ contracts to 
a small set concentrated around the maximum points of an explicitly given function.

In Section~\ref{sec:disc},
we present results concerning the case in which one or both of the functions $c,w$ have points of discontinuity. In this case, and depending on the nature of discontinuities of $c,w$, 
the solution $\alpha_{opt}$ 
can also have discontinuities, which correspond to atomic ($\delta$-function) components in the optimizing measure $\mu_{opt}$. This shows the necessity of our generalized formulation of the problem in terms of measures, 
allowing the solution to include such atomic components.
Our analytical results (Theorem \ref{th:discont})
partially characterize the solutions which arise. We also present results of numerical computations in several examples, which further clarify the phenomena which can occur.

The analytical developments leading to the proofs of the 
above-mentioned theorems are 
presented in Section \ref{sec:proofs}. We first derive a first-order condition, which is necessary and sufficient for  $\alpha\in {\cal{M}}(T,\bar{\eta})$ to be a solution of Problem \ref{prob:main} (Subsection \ref{sec:foc}). This condition plays a key role, as the study of its implications allows us both to prove the regularity (continuity, differentiability, absolute continuity) of the solution, under appropriate conditions, and to 
obtain explicit expressions for it.
These regularity proofs, given in sections \ref{sec:continuity}-\ref{sec:absolute}, require some delicate arguments, which we believe are of independent interest, and could be used in treating other problems. In particular, in order to prove the 
key result that the function
$\alpha_{opt}$, and hence the optimizing measure $\mu_{opt}$, is absolutely continuous under appropriate assumptions on $c,w$, we employ a known result from real analysis, which guarantees that 
a function which is differentiable apart from a countable set, with an $L^1$ derivative, is absolutely continuous. Note that {\it{a.e.}} differentiability of $\alpha_{opt}$,
which is immediate from the fact that
$\alpha_{opt}$ is monotone,
is {\it{not}} sufficient for such a result, as shown by the famous example of the Cantor function. Thus it is essential to prove that $\alpha_{opt}$ is differentiable at all points apart from a countable set, which we do in Subsection 
\ref{sec:differentiability}.

After obtaining the regularity results we turn to obtaining an explicit expression for the 
density $\alpha_{opt}^\prime(t)$ of the 
optimizing measure, on sets in which it is absolutely continuous (Subsections \ref{sec:explicit},\ref{sec:complete}), and to the study of the case in which $c$ or $w$ are discontinuous, which 
can lead to atomic components of the optimizing measure (Subsection \ref{sec:jumps}). Some concluding remarks and open questions are discussed in Section \ref{sec:discussion}.

\section{Existence and uniqueness of an optimizer}
\label{sec:existence}

Throughout this section, we will make the following assumption:

\begin{assumption}\label{a:eu}
	
\begin{itemize}
\item[(i)] $w,c\in L^1_+(T)$,
 and neither is identically $0$.

\item[(ii)] Either
\begin{equation}\label{eq:pnz1}
	c(t)>0\quad{\mbox{for {\it{a.e.}}}}\;t,
\end{equation}
or 
\begin{equation}\label{eq:pnz2}
	w(t)>0\quad{\mbox{for {\it{a.e.}}}}\;t.
\end{equation}
\end{itemize}
\end{assumption}

The key result of this section is the 
following existence and uniqueness theorem.

\begin{theorem}\label{th:existence}
	Under Assumption \ref{a:eu}, 
	there exists a unique solution 
	$\alpha_{opt}\in {\cal{M}}(T,\bar{\eta})$ of 
	Problem \ref{prob:main}.
\end{theorem}

To prove uniqueness, we will need the following lemma, which states that the functional
$\Phi$ defined by \eqref{eq:dPhi} is convex.

\begin{lemma}\label{lemma:convexity} 
Let $\alpha,\tilde{\alpha}\in {\cal{M}}(T,\bar{\eta})$, $\alpha\neq \tilde{\alpha}$, and define
\begin{equation}\begin{split}\label{eq:ae1}\alpha_{\epsilon}(t)&=(1-\epsilon)\alpha(t)+\epsilon \tilde{\alpha}(t),\\
\phi(\epsilon)&=\Phi[\alpha_{\epsilon}],\qquad\epsilon \in [0,1].\end{split}\end{equation}
Then $\phi''(\epsilon)>0$ for all
$\epsilon \in [0,1]$. Therefore, for $\epsilon\in (0,1)$,
\begin{equation}\label{eq:con1}
\Phi[(1-\epsilon)\alpha+\epsilon \tilde{\alpha}]<
(1-\epsilon)\Phi[\alpha]+\epsilon \Phi[\tilde{\alpha}].
\end{equation}
\end{lemma}

\begin{proof}
Note that $\alpha_{\epsilon}\in {\cal{M}}(T,\bar{\eta})$ for all
$\epsilon \in [0,1]$. 
Defining
$$S_{\epsilon}(t)={\cal{S}}[\alpha_{\epsilon}],$$
where ${\cal{S}}$ is given by \eqref{eq:spex1}
we have
$$S_\epsilon(t)=\int_0^T K(t,r) c(r) e^{\alpha_{\epsilon}(r)-\alpha_{\epsilon}(t)+\delta (r-t)}dr.$$
We have
$$\frac{d}{d\epsilon}\alpha_{\epsilon}(t)=v(t)\doteq \tilde{\alpha}(t)-\alpha(t),$$
so that
\begin{equation}\label{eq:des}
\frac{d}{d\epsilon}S_\epsilon(t)=\int_0^T K(t,r)c(r) e^{\alpha_{\epsilon}(r)-\alpha_{\epsilon}(t)+\delta (r-t)}(v(r)-v(t))dr,
\end{equation}
and
\begin{equation*}\begin{split}
\frac{d^2}{d\epsilon^2}S_\epsilon(t)&=\int_0^T K(t,r) c(r) e^{\alpha_{\epsilon}(r)-\alpha_{\epsilon}(t)+\delta (r-t)}(v(r)-v(t))^2dr\\
	&\geq \frac{1}{e^{[\bar{\eta}+\delta ]T}-1}\cdot \int_0^T c(r) e^{\delta r}(v(r)-v(t))^2dr.
\end{split}\end{equation*}
%
Hence,
\begin{equation*}\begin{split}
\phi''(\epsilon)&=\frac{d^2}{d\epsilon^2}\left[\int_0^T w(t)S_{\epsilon}(t)dt\right]=\int_0^T w(t)\left[\frac{d^2}{d\epsilon^2}S_\epsilon(t)\right] dt\\&\geq\frac{ 1}{e^{[\bar{\eta}+\delta ]T}-1}\cdot\int_0^T\int_0^T w(t)c(r) e^{\delta r}(v(r)-v(t))^2 dr dt\geq 0,
\end{split}\end{equation*}
and equality $\phi''(\epsilon)=0$ can hold only if 
\begin{equation}\label{eq:zz}
w(t)c(r)(v(r)-v(t))^2=0,\quad {\mbox{for {\it{a.e.}}}}\;t,r.\end{equation}
Assuming $\eqref{eq:pnz1}$, and 
fixing $t$ with $w(t)>0$, 
we conclude that $v(r)=v(t)$
for {\it{a.e.}} $r$, so
$\alpha,\tilde{\alpha}$ differ by a constant - but since $\tilde{\alpha}(0)=\alpha(0)=0$ this implies $\tilde{\alpha}=\alpha$, contradicting the assumption
$\tilde{\alpha}\neq\alpha$. We have thus shown that $\phi''(\epsilon)>0$. An analogous argument applied if
we assume \eqref{eq:pnz2}.
Thus $\phi$ is a strictly convex function, which implies 
\eqref{eq:con1}.
\end{proof}

\begin{lemma}\label{lemma:pc}
Assume that a sequence of functions 
\[\{\alpha_n\}_{n=1}^\infty\subset {\cal{M}}(T,\bar{\eta})\] converges 
pointwise to a function~$\alpha$. Then $\alpha\in {\cal{M}}(T,\bar{\eta})$, and we have
$$\lim_{n\rightarrow \infty}\Phi[\alpha_n]=\Phi[\alpha].$$
\end{lemma} 

\begin{proof}
The pointwise limit of a sequence of 
monotone functions is a monotone function, and by the pointwise convergence we have
$$\alpha(T)=\lim_{n\rightarrow\infty }\alpha_n(T)=\bar{\eta}T,$$
hence $\alpha\in {\cal{M}}(T,\bar{\eta})$. 

By the Lebesgue bounded convergence theorem, the pointwise convergence and uniform boundedness of $\alpha_n$, together with
\eqref{eq:spex1},
implies that ${\cal{S}}[\alpha_{n}]$  converge pointwise to 
${\cal{S}}[\alpha]$, and are uniformly bounded. Using
\eqref{eq:dPhi} and the Lebesgue bounded convergence theorem, this implies
$$\lim_{n\rightarrow \infty}\Phi[\alpha_{n}]=\lim_{n\rightarrow \infty}\int_0^T w(t){\cal{S}}[\alpha_{n}](t)dt=\int_0^T w(t){\cal{S}}[\alpha](t)dt=\Phi[\alpha].$$
\end{proof}

We now prove the existence and uniqueness of Problem~\ref{prob:main}.

\begin{proof}[Proof of Theorem \ref{th:existence}]
	We first prove existence.
	Let 
	$$\Phi_{\min}=\inf \{\Phi[\alpha]\;|\; \alpha\in {\cal{M}}(T,\bar{\eta})\}.$$
	Let $\{\alpha_n\}_{n=1}^\infty \subset {\cal{M}}(T,\bar{\eta})$ be a minimizing sequence,  that is 
	$$\lim_{n\rightarrow\infty}\Phi[\alpha_n]=\Phi_{\min}.$$
	Since~$\{\alpha_n\}_{n=1}^\infty$
	is a sequence of monotone non-decreasing functions, that are uniformly bounded, the Helly selection theorem
	(see e.g. \cite[Theorem 13.16]{carothers2000real}) implies 
	that there exists a subsequence 
	$\{\alpha_{n_k}\}_{k=1}^\infty$
	which converges pointwise to a function $\alpha_{opt}$.

By Lemma \ref{lemma:pc} we have 
$\alpha_{opt}\in {\cal{M}}(T,\bar{\eta})$
and $\Phi[\alpha_{opt}]=\Phi_{\min}$,
so that $\alpha_{opt}$ is a solution of Problem~\ref{prob:main}.
	
To see that~$\alpha_{opt}$ is {\it{unique}}, assume that $\alpha,\tilde{\alpha}\in {\cal{M}}(T,\bar{\eta})$ are two distinct solutions to problem \ref{prob:main}, so that~$\Phi[\alpha]=\Phi[\tilde{\alpha}]=\Phi_{\min}$.
	 Then, by \eqref{eq:con1} we have, for $\epsilon\in (0,1)$,
$$\Phi_{\min}\leq 
\Phi[(1-\epsilon)\alpha+\epsilon \tilde{\alpha}]<
(1-\epsilon)\Phi[\alpha]+\epsilon \Phi[\tilde{\alpha}]
=\Phi_{\min},$$
giving a contradiction. 	
\end{proof}

Having proved the existence and uniqueness of a solution, it now becomes interesting to characterize the solution as precisely as possible.

\section{Characterization of the optimizer}
\label{sec:char}

The optimal measure $\mu_{opt}$,
whose existence was proved in the 
previous section, can in principle be an arbitrary measure, so it can include an absolutely continuous and a singular component.
Our main results, to be presented in this section,
are aimed at obtaining a 
more precise characterization of the 
optimizer - including both its regularity properties and its explicit form.

Throughout this section, we will assume that Assumption \ref{a:eu} holds,
implying, by Theorem \ref{th:existence},
that Problem \ref{prob:main} has a unique solution, which is denoted by $\alpha_{opt}$. 
Further assumptions on $c,w$
will be made as needed for the various results.

\subsection{The case of continuous data}

The theorem below will guarantee that, under some regularity of the data (see Assumption \ref{a:reg}) the solution $\alpha_{opt}$ is
absolutely continuous, so that the measure $\mu_{opt}$ is absolutely continuous and has a
density $\eta_{opt}(t)=\alpha_{opt}^\prime(t)$.
Furthermore, we will find an {\it{explicit}} expression for this density at points of the support of the measure $\mu_{opt}$ (on points outside the support, this density is of course $0$).

We recall that the support of a
measure $\mu_{\alpha}$ ($\alpha\in {\cal{M}}(T)$), which we will denote by 
$\Omega(\alpha)$, is the set of points 
$t$  every open neighborhood of which has  positive measure, that is 
$$t\in \Omega(\alpha)\quad\Leftrightarrow\quad 
\alpha(t+\epsilon)-\alpha(t-\epsilon)>0\quad\forall \epsilon>0.$$
Note that the support of a measure is a closed set, and that in 
view of \eqref{eq:aper}, the 
set $\Omega(\alpha)$ is $T$-periodic in the sense that if
$t\in \Omega(\alpha)$ then $t+kT\in \Omega(\alpha)$ for all $k\in \Integer$. 
In terms of our model's interpretation, the support $\Omega(\alpha)$ is
the `active set' - the range of time at which some effort is expended.
$|\Omega(\alpha)\cap [0,T)|$ will denote the Lebesgue measure of this set of times within each period. 

For the following theorem, we will make the following assumption on the regularity of our data:

\begin{assumption}\label{a:reg}
	
\begin{itemize}	
\item[(i)] $w,c\in C(T)$ 
are absolutely continuous on $[0,T]$ (hence on any closed interval).

\item[(ii)] $w,c$ are differentiable at all points apart from a countable set.

\item[(iii)] $c(t)>0$ and $w(t)>0$ for all $t$.
\end{itemize}
\end{assumption}

\begin{theorem}\label{th:main}
Under Assumption \ref{a:reg},
let $\alpha_{opt}$ be the 
solution of Problem 
\ref{prob:main}, and denote its support by
$\Omega=\Omega(\alpha_{opt})$.
Then $\alpha_{opt}$ is absolutely continuous on any closed interval, and 
\begin{equation}\label{eq:alphae}\alpha_{opt}(t)=\int_0^t \eta_{opt}(t)dt,\end{equation}
where
\begin{equation}\label{eq:etaopt}\eta_{opt}(t)=\begin{cases}
\eta^*(t) & t\in \Omega\\
0 & t\not\in \Omega,
\end{cases}\end{equation}
with
\begin{equation}\label{eq:etaopt1}
	\eta^*(t)\doteq \frac{1}{\lambda}\cdot \sqrt{w(t)c(t)} -\frac{1}{2} \left[\ln \frac{c(t)}{w(t)} \right]'-\delta,
\end{equation}
\begin{equation}\label{eq:lambda1}\lambda =\frac{\int_{\Omega\cap [0,T]} \sqrt{w(t)c(t)}dt}{T\bar{\eta}+\delta\cdot |\Omega\cap [0,T)|+\frac{1}{2} \int_{\Omega\cap [0,T]}\left[\ln \frac{c(t)}{w(t)}\right]'dt}.\end{equation}
We also have
\begin{equation}\label{eq:cc}{\cal{S}}[\alpha_{opt}](t)=\lambda\cdot \sqrt{\frac{c(t)}{w(t)}},\;\;\;t\in \Omega.\end{equation}
\end{theorem}

\begin{figure}
	\begin{center}
		\includegraphics[width=1\linewidth]{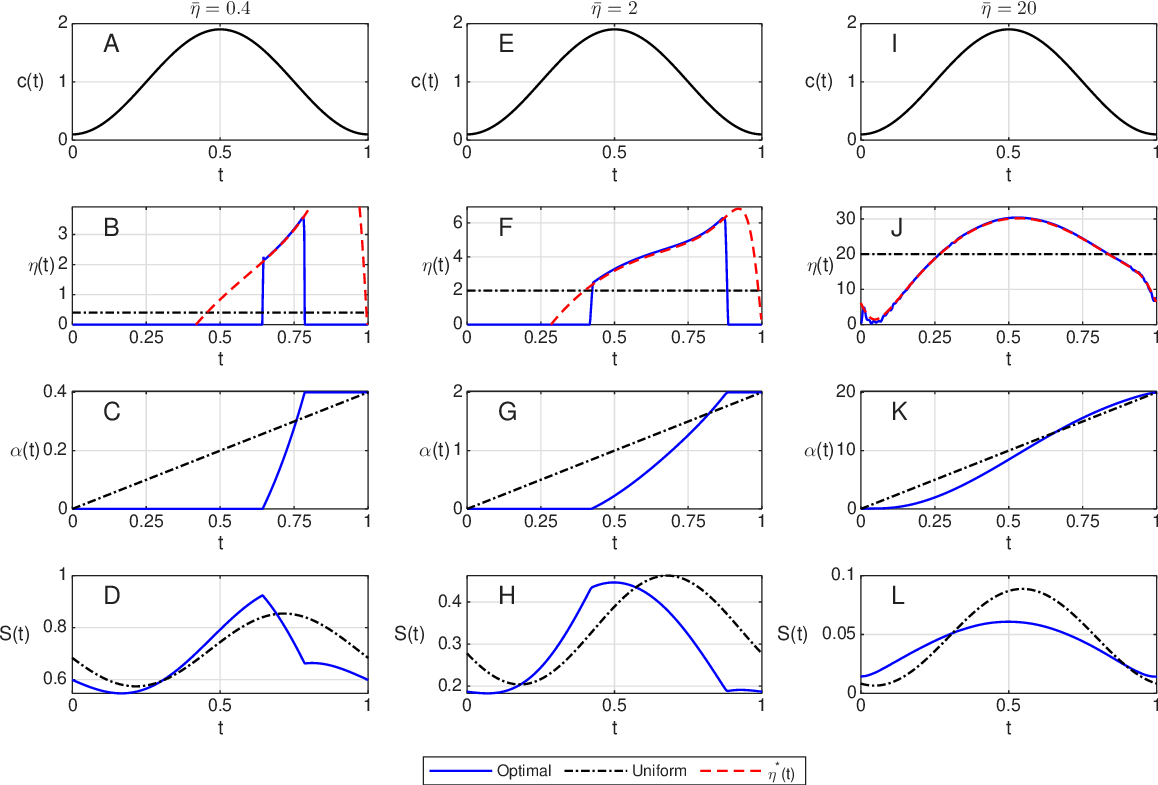}
	\end{center}
	\caption{Optimal effort profile ($\eta_{opt}(t)$), cumulative effort
		profile ($\alpha_{opt}(t)$) and the resulting function $S_{opt}(t)$
		for inflow rate $c(t)=1-0.9\cos(2\pi t)$, weight $w(t)\equiv 1$, $\delta=1$,  and three values of $\bar{\eta}$. For comparison, dashed black lines show corresponding results for a uniform effort profile $\eta(t)=\bar{\eta}$. The dashed red line in the second row display the analytical expression \eqref{eq:etaopt1}.}
	\label{fig:1}
\end{figure}

The proof of Theorem \ref{th:main} will be given at the end of Subsection
\ref{sec:explicit}, after developing the necessary tools.
In fact the above results  will be shown to hold even if the functions $c,w$ are allowed to be discontinuous on a discrete set, restricting to the points of
continuity. This is treated in Theorem \ref{th:discont} below. 

Theorem \ref{th:main} provides an explicit expression
\eqref{eq:alphae}, \eqref{eq:etaopt} for the the functional form of the solution
of Problem \ref{prob:main}, but lacks a detailed description of
the set $\Omega(\alpha_{opt})$ on which the 
function $\eta_{opt}$ is supported.
Therefore, in general, obtaining
the optimizer will require a
numerical approach. The numerical computations were performed by direct minimization of a discretized version of the functional $\Phi$. Details of the numerical method are given in Appendix \ref{sec:numerical}.

We now present some numerical results which will illustrate the contents of Theorem \ref{th:main}.

\begin{example}\label{ex:1}
In Figure \ref{fig:1} we use the inflow rate
$c(t)=1-0.9\cos(2\pi t)$ (graph shown in the top row) with period $T=1$, constant weight $w(t)\equiv 1$, and $\delta=1$. 
\end{example}

Since the functions 
$c(t),w(t)$ satisfy all assumptions of Theorem \ref{th:main}, the solution $\alpha_{opt}$ of Problem \ref{prob:main} is absolutely continuous.

We numerically compute the optimal cumulative effort 
$\alpha_{opt}(t)$  and the corresponding effort profile $\eta_{opt}(t)=\alpha_{opt}^\prime(t)$ and display their graphs (solid blue line). For comparison, the dashed black lines display the constant effort profile $\eta(t)=\bar{\eta}$, with the same mean, and the corresponding cumulative effort profile.
 We also present the resulting function $S(t)={\cal{S}}[\alpha_{opt}]$, both for the optimal effort profile $\alpha_{opt}$ and for the constant effort profile.  

Figure \ref{fig:1} presents results for three 
different values of the mean effort per period 
$\bar{\eta}$. In the graphs
displaying $\eta(t)$, we also 
display the graph of the analytical expression $\eta^*(t)$ given by \eqref{eq:etaopt1} (dashed red line). In accordance with Theorem \ref{th:main}, we observe that $\eta_{opt}(t)$ coincides with $\eta^*(t)$ at all points 
$t$ in the support $\Omega(\alpha_{opt})$.

When $\bar{\eta}$ is low (left column), we observe that the support 
$\Omega(\alpha_{opt})$ is a narrow interval - all effort is concentrated in a short time period within each cycle, leading to a rapid reduction in the pollutant concentration $S$ during this period. We note that this 
active time-period commences somewhat after the time of peak inflow rate $c(t)$, allowing 
pollution to accumulate to a relatively high level so that each unit of effort produces a
large effect. 

As $\bar{\eta}$ increases, so that the
constraint on total effort is relaxed, we observe (Figure \ref{fig:1}, middle column) 
that the active interval $\Omega(\alpha_{opt})$ expands. When
$\bar{\eta}$ is sufficiently high (right column), the function $\eta_{opt}$ becomes positive at all times, that is, $\Omega(\alpha_{opt})=\Real$, so that there are no `holidays' and some effort is expended at any time. This phenomenon is a general one, as proved in the subsequent subsection. 

\subsection{The case of a positive optimizer}

When $\bar{\eta}$ is sufficiently large (larger than an explicitly given value $\bar{\eta}_m$ given below), 
$\eta_{opt}$ is everywhere positive, that is, $\Omega(\alpha_{opt})=\Real$,
in which case \eqref{eq:etaopt}
provides an entirely explicit
expression for the solution.

\begin{theorem}\label{th:pos} Under Assumption \ref{a:reg}, define:
\begin{equation}\label{eq:etam}\bar\eta_m\doteq  \frac{1}{T}\int_0^T \sqrt{w(t)c(t)}dt\cdot \sup_{t\in [0,T)}\left[\frac{1}{\sqrt{w(t)c(t)}}\left(\delta+\frac{1}{2}\cdot \left[\ln\frac{c(t)}{w(t)}\right]'\right)\right]-\delta.
\end{equation}
Then,
	
	(i) If $\bar{\eta}>\bar{\eta}_m$, then $\Omega(\alpha_{opt})=\Real$ and
	$\alpha_{opt}$  is given by
\begin{equation}\label{eq:in1}\alpha_{opt}(t)=\int_0^t \hat{\eta}(s)ds,\end{equation}
	where $\hat{\eta}(t)$ is the positive function defined by
	\begin{equation}\label{eq:hateta}
		\hat{\eta}(t)\doteq \frac{\bar{\eta}+\delta}{\ \frac{1}{T}\int_0^T \sqrt{w(s)c(s)}ds}\cdot \sqrt{w(t)c(t)}-\frac{1}{2}\cdot \left[\ln\frac{c(t)}{w(t)}\right]'-\delta.\end{equation}
	
	In this case we also have
	\begin{equation}\label{eq:shat}
		S_{opt}(t)=
		\frac{1}{\bar{\eta}+\delta}\cdot \frac{1}{T}\int_0^T \sqrt{w(s)c(s)}ds\cdot \sqrt{\frac{c(t)}{w(t)}},
	\end{equation}
	and the optimal value of Problem \ref{prob:main}
	is given by
	\begin{equation}\label{eq:optval}
		\Phi_{\min}=\Phi[\alpha_{opt}]= \frac{1}{\bar{\eta}+\delta}\cdot\frac{1}{T}\cdot \left(\int_0^T  \sqrt{w(s)c(s)} ds\right)^2.
	\end{equation}

	(ii) If $\bar{\eta}<\bar{\eta}_m$ then $\Omega(\alpha_{opt})$ is a set of positive measure.
\end{theorem}

The proof of Theorem \ref{th:pos} will
be given in Subsection \ref{sec:complete}.

The above theorem thus shows that when the total amount of effort that can be expended per period ($\bar{\eta}T$) is sufficiently large, the optimal solution will involve expending effort at all times, while if the limitation on the 
total effort is sufficiently stringent, the optimal solution will include some 
`holidays' in which effort is not expended. It should be stressed that in the latter case, while on its support the function $\eta_{opt}$ will be equal to the function 
$\eta^*$ given by \eqref{eq:etaopt1},
it is definitely 
{\it{not}} true that its support 
is equal to the set of points where
$\eta^*$ is positive, as can be observed in Figure \ref{fig:1}.

\begin{example}
In the case
$c(t)=1-0.9\cos(t),w(t)\equiv 1,\delta=1$
shown in Figure \ref{fig:1}, we can calculate, using \eqref{eq:etam},
that $\bar{\eta}_m\approx 16.37$, consistent with the results presented in the figure.
\end{example}


With regard to the expression
\eqref{eq:optval} for $\Phi_{\min}$, let us note that applying the Cauchy-Schwartz inequality we have
\begin{equation*}
\begin{split}
\Phi_{\min}&=\frac{1}{\bar{\eta}+\delta}\cdot\frac{1}{T}\cdot\left(\int_0^T  \sqrt{w(t)c(t)} dt\right)^2\\&\leq\frac{1}{\bar{\eta}+\delta}\cdot\frac{1}{T}
\left(\int_0^T w(t)dt\right)\left(\int_0^T c(t)dt\right)=\left(\int_0^T w(t)dt\right)\frac{\bar{c}}{\bar{\eta}+\delta},
\end{split}\end{equation*}
with equality if and only $\frac{c(t)}{w(t)}$ is a constant function - in which case $\eta_{opt}=\bar{\eta}$ is constant. In particular,  assuming $w$ is constant,
the right-hand side of the inequality
is the value achieved when $c(t)=\bar{c}$ is constant. This shows that, assuming $\bar{\eta}>\bar{\eta}_m$, a periodic inflow $c(t)$ allows us to achieve a lower value of $\Phi$, compared to a constant inflow with the
same mean $\bar{c}$, {\it{provided}}
that we choose the effort profile
$\eta(t)$ in an optimal way.

\subsection{Dependence of the optimizer on $\bar{\eta}$ - limits of high and low total effort}
\label{sec:lh}

The constraint $\alpha\in\mathcal{M}(T,\bar\eta)$ in
Problem \ref{prob:main} requires that the total effort expended per period is $T\bar{\eta}$. It is thus of interest to examine how the value of $\bar{\eta}$ affects the optimal effort profile. To do so, we will now denote 
by $\alpha_{opt,\bar{\eta}}(t)$ the solution of Problem \ref{prob:main} corresponding to a specific value $\bar{\eta}$.

In the case where $\bar{\eta}\rightarrow \infty$, 
we know from Theorem \ref{th:pos} that, under its assumptions, the solution is given by \eqref{eq:in1}, \eqref{eq:hateta} for $\bar{\eta}>\bar{\eta}_m$. This immediately implies 
\begin{proposition}
Under the assumptions of Theorem 
\ref{th:pos} we have
$$\lim_{\bar{\eta}\rightarrow \infty}\frac{\alpha_{opt,\bar{\eta}}(t)}{T\bar{\eta}}=\frac{\int_0^t \sqrt{w(s)c(s)}ds}{\int_0^T \sqrt{w(s)c(s)}ds}$$
uniformly in $t$.
\end{proposition}

We now consider the case in which the total allowed effort per period is small, $\bar{\eta}\rightarrow 0$. 
Examining Figure \ref{fig:1}, we observe that for small $\bar{\eta}$
the active set $\Omega(\alpha_{opt,\bar{\eta}})$ becomes 
a small interval, so all effort is concentrated within a short time duration in each period. The next result, which is proved in Subsection \ref{sec:e0} shows that this is indeed the case in general, and explicitly identifies the location of the small time-interval in which effort is concentrated.

\begin{theorem}\label{prop:eta0}
	Under Assumption \ref{a:reg}, define the function
\begin{equation}\label{eq:deff}
	f(t)=\int_0^t c(r)e^{\delta r}dr\int_t^T w(r)e^{-\delta r } dr-e^{-\delta T}\int_0^t w(r)e^{-\delta r }dr\int_t^T c(r)e^{\delta r}dr,
	\end{equation}
	which is $T$-periodic. Let 
	$$f_{\max}=\max_{t\in \Real} f(t),$$
	and, for $\epsilon>0$,
	$$M_{\epsilon}=\{ t\in \Real \;|\; f(t)>f_{\max}-\epsilon\}.$$
	Then, for any $\epsilon>0$, there 
	exists a value $\bar{\eta}(\epsilon)$
	such that 
	$$\bar{\eta}\in (0,\bar{\eta}(\epsilon))\quad\Rightarrow\quad \Omega(\alpha_{opt,\bar{\eta}})\subset M_\epsilon.$$
\end{theorem}
Note that, generically, the function $f$ defined by \eqref{eq:deff} will have a unique global maximum point $t_{max}$ in $[0,T)$,
and as $\epsilon\rightarrow 0$ the
set $M_{\epsilon}$ will consist of 
small intervals containing the points $t_{max}+kT$ ($k\in\Integer$).

\begin{example}
In Example \ref{ex:1}, an explicit computation and numerical maximization of the function $f$, shows that $t_{max}\approx 0.725$, consistent with the results seen in Figure \ref{fig:1}.
\end{example}

\subsection{The effect of discontinuities in $c(t)$, $w(t)$}
\label{sec:disc}
In Theorem~\ref{th:main}, continuity of the data
($c(t),w(t)$) was assumed.  Here, we relax this assumption and investigate the implications of discontinuities in these data. As we shall see, some such discontinuities can lead to discontinuity of the optimizer 
$\alpha_{opt}$, which corresponds to atomic (delta-function) components of the optimizing measure $\mu_{opt}$, so that in this case the measure is no longer 
absolutely continuous.

We relax Assumption \ref{a:reg} and
make the following

\begin{assumption}\label{a:pc}
	
\begin{itemize}
\item[(i)] The right and left-hand limits
$c(t\pm),w(t\pm)$ exist at every point $t\in \Real$, with
\begin{equation}
\label{eq:pp1}
c(t\pm)>0,\quad w(t\pm)>0,\quad\forall t.
\end{equation}

\item[(ii)] Defining
\begin{equation}\label{eq:varphi}
\varphi(t)=\frac{c(t)}{w(t)},
\end{equation}
and defining the set 
$$D=\{ t \;|\; \varphi(t+)\neq \varphi(t-)\}$$ 
of discontinuity points of $\varphi$ (which are necessarily points of discontinuity of at least one of the functions $w(t),c(t)$), we have that the set $D\cap [0,T)$ is finite.

\item[(iii)] $c(t),w(t)$ are differentiable at all points $t$
of the open set $U=\Real\setminus D$ except for a countable set, and absolutely continuous on every closed interval $[a,b]\subset U$.
\end{itemize}
\end{assumption}

We divide the set of discontinuity 
points into two subsets: the upward jumps $D_+$ and the downward jumps $D_-$:
\begin{equation}\label{eq:D}
D_+=\{ t \;|\; \varphi(t+)> \varphi(t-)\},\quad D_-=\{ t \;|\; \varphi(t+)< \varphi(t-)\}.
\end{equation}

\begin{theorem}\label{th:discont}
Under Assumption \ref{a:pc},
$\alpha_{opt}$ is differentiable at all points of $U=\Real\setminus D$ apart from a countable set, is 
absolutely continuous on every closed sub-interval of $U$, and we have:

(i) For {\it{a.e.}} $t\in U$:
	\begin{equation}\label{eq:ef1}\alpha_{opt}^\prime(t)=\begin{cases}
		\frac{1}{\lambda}\cdot \sqrt{w(t)c(t)} -\frac{1}{2} \left[\ln \frac{c(t)}{w(t)} \right]'-\delta & t\in U\cap \Omega(\alpha_{opt})\\
		0 & t\in U,\;t\not\in \Omega(\alpha_{opt})
	\end{cases},\end{equation}
where $\lambda>0$ is a constant.

(ii) If $t\in D_+$ then $t\not\in 
\Omega(\alpha_{opt})$.

(iii) If $t\in D_-$ and
if $t$ is a limit point of the set
of $\Omega(\alpha_{opt})$ both from the right and from the left, that is
\begin{equation}\label{eq:ts}t\in \overline{\Omega(\alpha_{opt})\cap (t,\infty)}\cap \overline{\Omega(\alpha_{opt})\cap (-\infty,t)},\end{equation}
then $t\in \Omega(\alpha_{opt})$ and $\alpha_{opt}$ has a jump discontinuity at $t$, given by 
\begin{equation}\label{eq:js}
	\alpha_{opt}(t+)-\alpha_{opt}(t-)
	=\frac{1}{2}\ln\left(\frac{\varphi(t-)}{\varphi(t+)} \right)>0.
\end{equation}
At such point, the function 
$S_{opt}={\cal{S}}[\alpha_{opt}]$
has a downward jump discontinuity, with
\begin{equation}\label{eq:sjump}
\frac{S(t+)}{S(t-)}=\sqrt{\frac{\varphi(t_+)}{\varphi(t-)}}<1.
\end{equation}
\end{theorem}
The proof of Theorem \ref{th:discont}
will be given in Subsection \ref{sec:jumps}.

The above theorem implies that the measure 
$\mu_{opt}$ has 
both an absolutely continuous component, given by \eqref{eq:ef1}, and, possibly, an atomic component supported at points of $D_-$.
It should be noted that the
characterization of the discontinuity points of
$\alpha_{opt}$ given in the above theorem is incomplete.
The theorem does not determine when the condition \eqref{eq:ts} - which says that $t$ is a limit point of points in 
$\Omega(\alpha_{opt})$ both from the right and from the left - holds. Moreover,
it does not determine what happens at points of 
$t\in D_-$ for which
\eqref{eq:ts} does not hold.

To illustrate the results of Theorem 
\ref{th:discont}, as well as to obtain details of the structure of the optimizer which the theorem does not provide and observe some interesting phenomena, we carry out several numerical investigations, using discontinuous inflow rate functions $c(t)$.

\begin{figure}
	\begin{center}
		\includegraphics[width=1\linewidth]{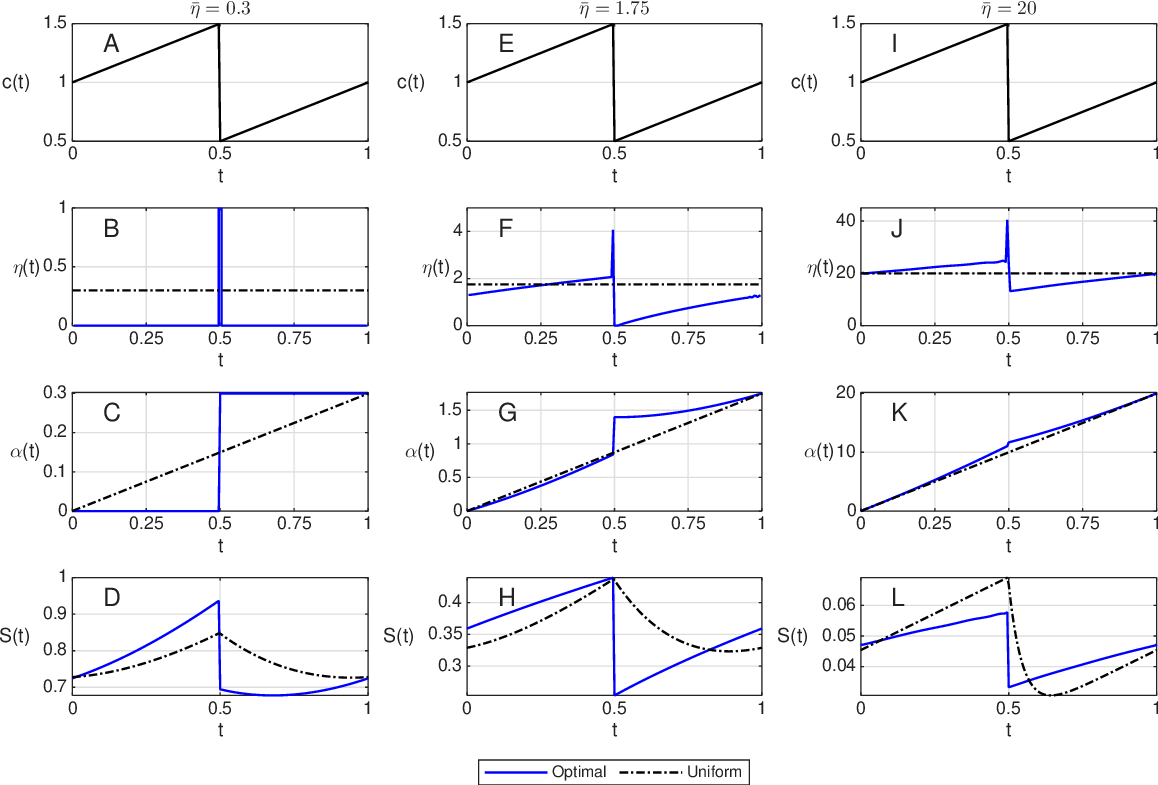}
	\end{center}
	\caption{Optimal effort profile ($\eta_{opt}(t)$), cumulative effort
		profile ($\alpha_{opt}(t)$) and the resulting function $S_{opt}(t)$
		for inflow rate $c(t)$ given by \eqref{eq:ceeo}, weight $w(t)\equiv 1$, $\delta=1$,  and three values of $\bar{\eta}$.}
	
	\label{fig:2}
\end{figure}

\begin{figure}
	\begin{center}
		\includegraphics[width=1\linewidth]{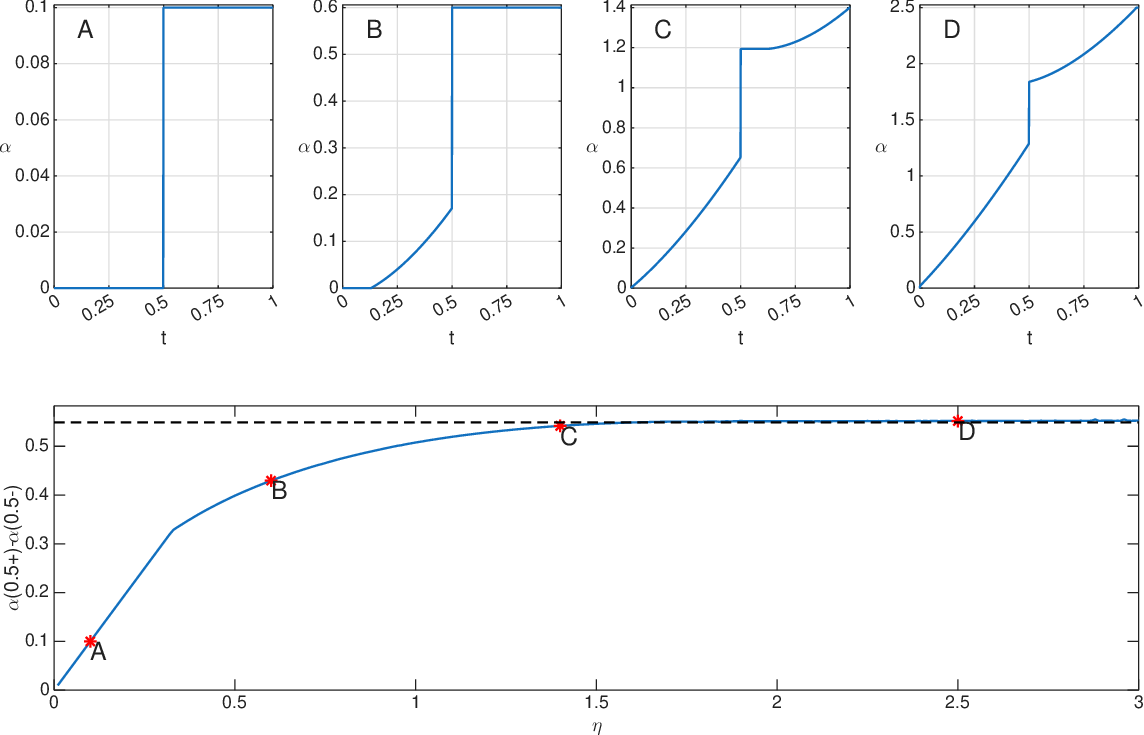}
	\end{center}
	\caption{Optimal cumulative effort
		profile ($\alpha_{opt}(t)$) 
		for inflow rate given by \eqref{eq:ceeo}, weight $w(t)\equiv 1$, $\delta=1$,  and four values of $\bar{\eta}$. The bottom panel shows the size of the jump in $\alpha_{opt}$ occurring at $t=0.5$.}
	\label{fig:2.5}
\end{figure}

\begin{example}
In figure \ref{fig:2} we take the
inflow rate function $c(t)$ as the `sawtooth' function which is the
$1$-periodic extension of 
\begin{equation}\label{eq:ceeo}
	c(t)=\begin{cases}
		t+1 & 0\leq t<0.5\\
		t & 0.5 \leq t<1.
	\end{cases}
\end{equation}
We take $w(t)\equiv 1$, so that
$\varphi(t)=c(t)$, and $\delta=1$.
\end{example}
The discontinuity
points of $\varphi$, of the two types, are
$$D_+=\emptyset,\quad D_-=\{0.5+k\;|\; k\in \Integer\}.$$
For all three values of 
$\bar{\eta}$ displayed, we observe that the solution 
$\alpha_{opt}$ has a jump
discontinuity at the discontinuity point $t=0.5$ of 
$c$, which corresponds to a 
$\delta$-function component at this point for the corresponding measure $\mu_{opt}$. 
We also observe two interesting transitions in the shape of the optimizer.
These may be examined more closely in Figure \ref{fig:2.5}, whose bottom panel
displays the size of the jump discontinuity of $\alpha_{opt}$ at $t=0.5$ (that is the mass of the $\delta$-function component of $\mu_{opt}$ located at this point) as a function of $\bar{\eta}$, and whose upper part shows the
$\alpha_{opt}$ profiles corresponding to
the four points marked point on the graph.
When $\bar{\eta}$ is sufficiently small, as in $A$, we have $\Omega(\alpha_{opt})\cap [0,1)=\{0.5\}$, that is the 
optimizing measure consists of pure $\delta$-functions at the points of $D_-$. For a range of intermediate values of $\bar{\eta}$, as in $B,C$, the measure has both atomic components and an absolutely continuous one, with
$\Omega(\alpha_{opt})\cap [0,1)$ being an interval whose right end-point is $t=0.5$, at which the $\delta$-function is located. For yet larger values of $\bar{\eta}$, as in $D$, $t=0.5$ becomes an interior point of $\Omega(\alpha_{opt})\cap [0,1)$. It is only in this third case that the condition \eqref{eq:ts} holds, which implies the validity of the formula \eqref{eq:js} for the 
size of jump in $\alpha$ (that is, the mass of the atomic component) at $t=0.5$, which gives
$$\alpha_{opt}(0.5+)-\alpha_{opt}(0.5-)=
\frac{1}{2}\ln\left(\frac{c(0.5-)}{c(0.5+)} \right)=\frac{1}{2}\ln(3)\approx 0.549,$$
the value shown as a dashed horizontal line in the bottom graph of Figure \ref{fig:2.5}.
We note that our analytic results do not
provide an expression for the size of the jump in the cases $B,C$, nor for the minimal value of $\bar{\eta}$ for which \eqref{eq:ts} holds.

\begin{figure}
	\begin{center}
		\includegraphics[width=1\linewidth]{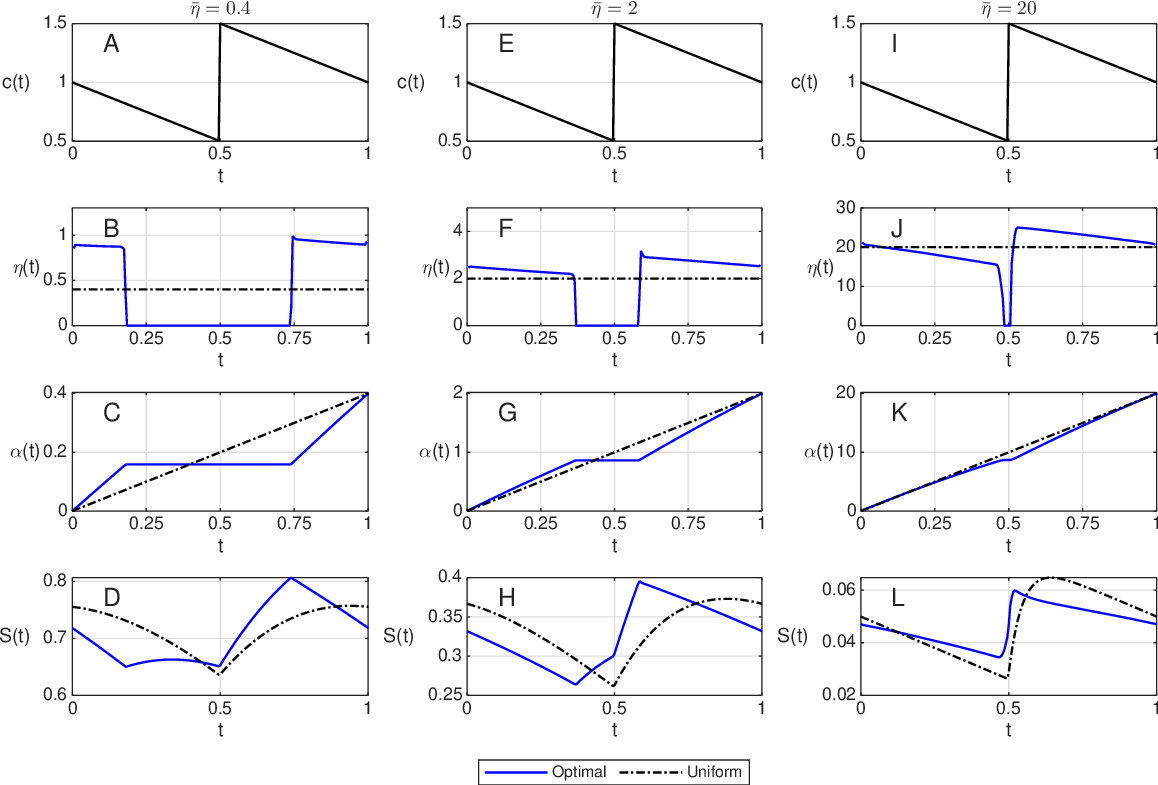}
	\end{center}
	\caption{Optimal effort profile ($\eta_{opt}(t)$), cumulative effort
		profile ($\alpha_{opt}(t)$) and the resulting function $S_{opt}(t)$
		for inflow rate $c(t)$ given by \eqref{eq:ceeo1}, weight $w(t)\equiv 1$, $\delta=1$,  and three values of $\bar{\eta}$. }
	\label{fig:3}
\end{figure}

\begin{example}
In Figure \ref{fig:3} we take the
inflow rate function $c(t)$ as the 
$1$-periodic extension of the function
\begin{equation}\label{eq:ceeo1}
	c(t)=\begin{cases}
		1-t & 0\leq t<0.5\\
		2-t & 0.5 \leq t<1,
	\end{cases}
\end{equation}
\end{example}
In this case
$$D_+=\{0.5+k\;|\; k\in \Integer\},\qquad D_-=\emptyset.$$
In agreement with part (ii)
of Theorem \ref{th:discont}, we have
$0.5\not\in \Omega(\alpha_{opt})$, and since $c(t)$ 
is continuous for $t\neq D_+$, the optimizing measure does not have an atomic component.

\begin{figure}
	\begin{center}
		\includegraphics[width=1\linewidth]{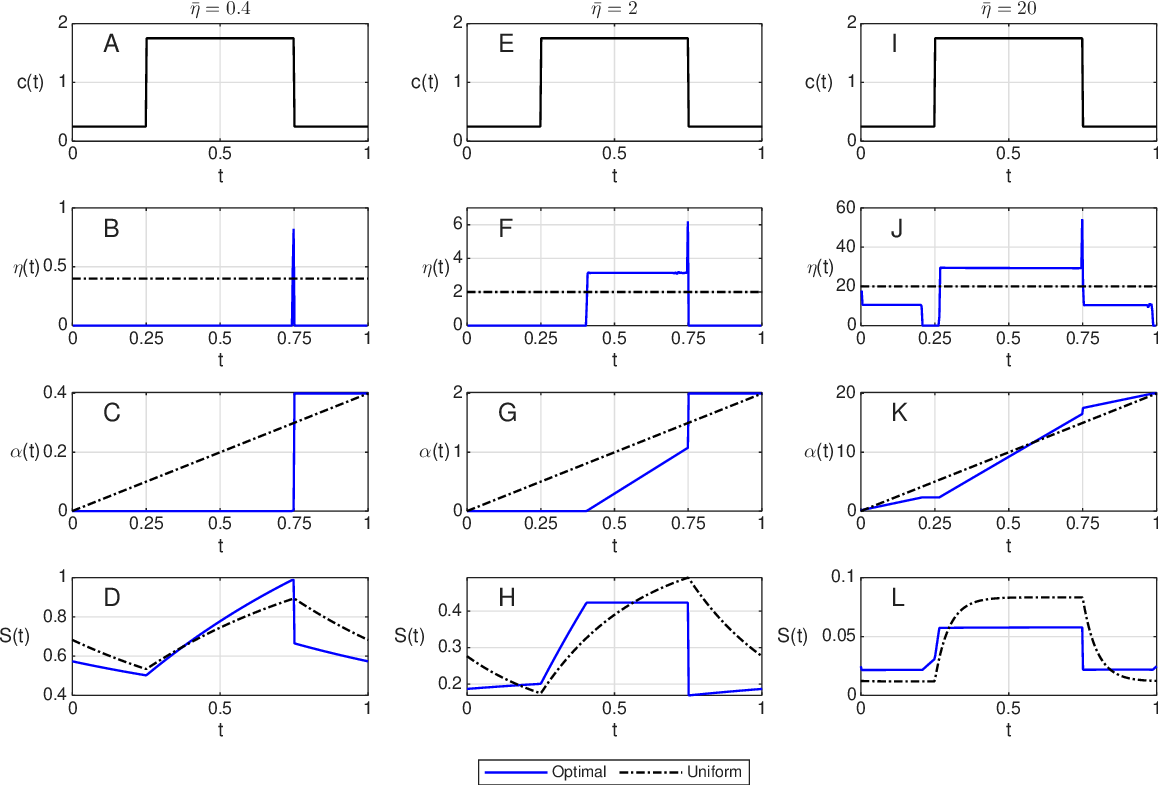}
	\end{center}
	\caption{Optimal effort profile ($\eta_{opt}(t)$), cumulative effort
		profile ($\alpha_{opt}(t)$) and the resulting function $S_{opt}(t)$
		for inflow rate $c(t)$ given by \eqref{eq:cee}, weight $w(t)\equiv 1$, $\delta=1$,  and three values of $\bar{\eta}$. }
	\label{fig:4}
\end{figure}

\begin{example}
In Figure \ref{fig:4} we take the 
inflow rate function $c(t)$ as the 
$1$-periodic extension of the function
\begin{equation}\label{eq:cee}
c(t)=\begin{cases}
0.25 & 0\leq t<0.25\\
1.75 & 0.25\leq t<0.75\\
0.25 & 0.75\leq t<1
\end{cases}
\end{equation}
We take $w(t)\equiv 1$, so that
$\varphi(t)=c(t)$.
\end{example}
The discontinuity
points of $\varphi$, of the two types, are
$$D_+=\{0.25+k\;|\; k\in \Integer\},\quad D_-=\{0.75+k\;|\; k\in \Integer\}.$$
Part (ii) of Theorem \ref{th:discont} 
tells us that $0.25\not\in \Omega(\alpha_{opt})$, and in figure
\ref{fig:4} we can see that this is indeed the case for the three values of $\bar{\eta}$ 
for which results are displayed: there is 
always an interval around the point $t=0.25$
in which 
$\alpha_{opt}(t)$ is constant. Note that the length of this interval decreases as 
$\bar{\eta}$ increases. 
At the second discontinuity point $t=0.75$, the optimizing measure always has an atomic ($\delta$-function) component
located at $t=0.75$, which 
generates a jump discontinuity in the graph of $\alpha_{opt}$. 
For low values of 
$\bar{\eta}$, (as panel A), the optimizing 
measure is the $1$-periodic extension of a pure $\delta$-function, located at $t=0.75$. For intermediate values of $\bar{\eta}$
the optimizing measure consists of both  atomic components at $t=0.75+k$ ($k\in \Integer$), and an absolutely continuous component, supported on  intervals
whose right end-points are $t=0.75+kT$ ($k\in \Integer$), which expand as 
$\bar{\eta}$ increases. In these intervals,
the function $\eta_{opt}(t)$
is constant and positive, according to the
expression \eqref{eq:ef1}, and the fact that $c,w$ are constant in $[0.25,0.75]$. Thus in this regime, one starts expending effort during the period in which
inflow rate is high, ending with a
concentrated effort ($\delta$-function) when the inflow rate drops at $t=0.75$, subsequently going on a holiday. For high values of $\bar{\eta}$, the support $\Omega(\alpha_{opt})$ is an interval extending to the left and to the right of the point $t=0.75$, so that effort is expended both before and after time
$t=0.75$, at which a concentrated effort is made.
Note that only in this case the condition \eqref{eq:ts} holds, so that
part (iii) of Theorem implies that $t=0.75$ is a
discontinuity point of $\alpha_{opt}$, and the
size of the jump is given by
\eqref{eq:js}. 

In both Figure \ref{fig:2} 
and Figure \ref{fig:4}, we 
observe that for sufficiently low values of $\bar{\eta}$, the optimizing measure consists 
purely of $\delta$-functions located at the downward-jump points of $\varphi$. We can prove a general result showing that this phenomenon occurs under certain circumstances. 

\begin{proposition}\label{prop:sd}
	Assume that $w(t)\equiv 1$ and that there is a single point $t^*\in (0,T]$ such that $c(t)$ is strictly monotone increasing on $(0,t^*)$ and on $(t^*,T)$, and has a downward jump at $t^*$: 
	$c(t^*+)<c(t^*-)$. 
    Then, for sufficiently small $\bar{\eta}$, we have
\begin{equation}\label{eq:aj}\alpha_{opt}(t)=\begin{cases}
		0 & 0\leq  t<   t^*\\
		\bar{\eta}T & t^*\leq t\leq T,
	\end{cases}\end{equation}
	so that $\mu_{opt}$ consists 
	of $\delta$-functions located at $t^*+kT$ ($k\in \Integer$).
\end{proposition}

Note that the example of 
Figure \ref{fig:2} satisfies the hypotheses of the above proposition. The proof of this result, to be given in Subsection \ref{sec:jumps}, is interesting in that we show that, under the stated conditions, the function 
given by \eqref{eq:aj} satisfies the first-order condition given by Lemma \ref{lemma:bh1}, which is necessary and sufficient for optimality. This demonstrates how, in case that we have a conjecture for the optimizer, we can verify it by using the first-order condition.
The condition that $c(t)0$ is monotone increasing apart 
from the point of discontinuity is essential in the above proposition, as is shown by

\begin{figure}
	\begin{center}
		\includegraphics[width=1\linewidth]{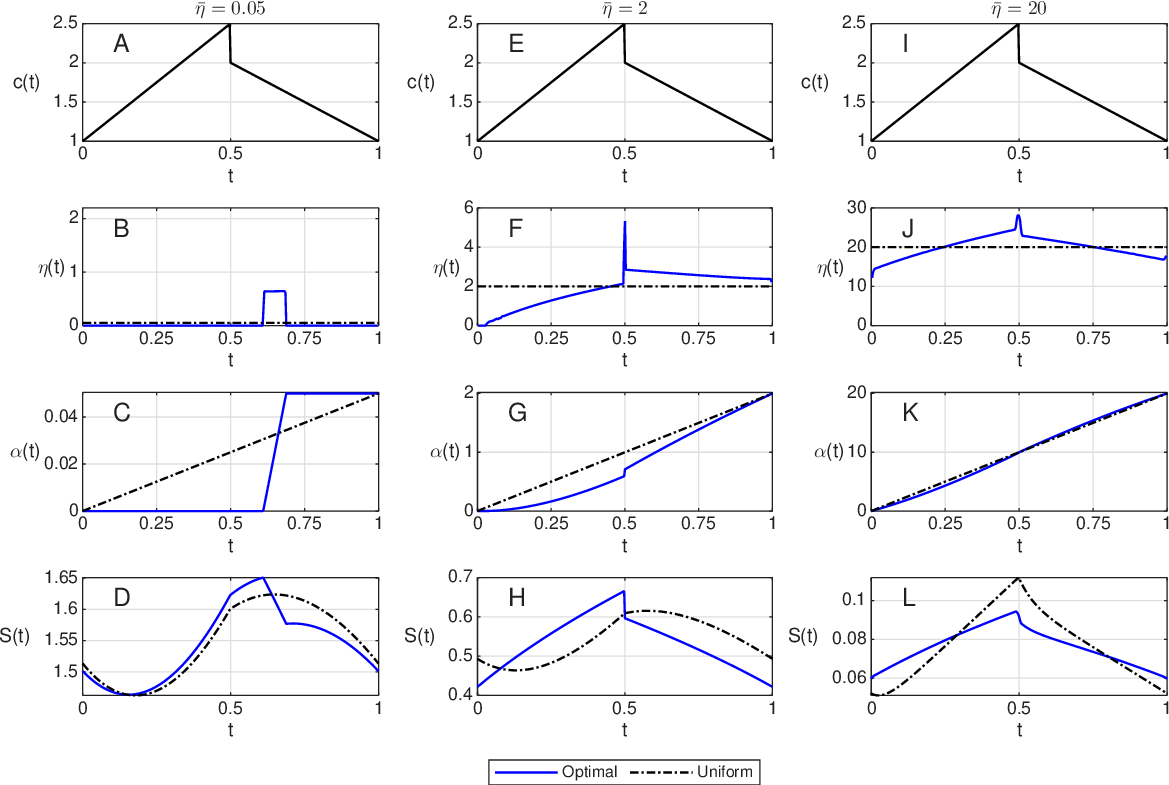}
	\end{center}
	\caption{Optimal effort profile ($\eta_{opt}(t)$), cumulative effort
		profile ($\alpha_{opt}(t)$) and the resulting function $S_{opt}(t)$
		for inflow rate $c(t)$ given by \eqref{eq:ce2}, weight $w(t)\equiv 1$, $\delta=1$,  and three values of $\bar{\eta}$. }
	\label{fig:6}
\end{figure}

\begin{example}
In  Figure \ref{fig:6} we take
\begin{equation}\label{eq:ce2}c(t)=\begin{cases}
	1+3t & 0\leq t\leq 0.5\\
	3-2t & 0.5\leq t\leq 1
\end{cases},
\end{equation}
and $w(t)\equiv 1$.
\end{example}
In this 
example, although $t^*=0.5$ is 
a downward jump point, we see that 
for $\bar{\eta}$ sufficiently small
the point $t^*$ is outside 
$\Omega(\alpha_{opt})$, so that $\alpha_{opt}$ is absolutely continuous. Thus the conclusion of 
Proposition \ref{prop:sd} does not hold. Note that for sufficiently large $\bar{\eta}$, $\Omega(\alpha_{opt})$ {\it{does}} include
$t^*$, and the optimizing measure then has a $\delta$-function component located at $t^*$. Thus here the atomic component is `born' at some positive value of $\bar{\eta}$

\section{Characterization of the optimizer: proofs}
\label{sec:proofs}

This section is devoted to the proof of
Theorems \ref{th:main}-\ref{th:discont}.
We assume throughout that
Assumption \ref{a:eu} (see Subsection \ref{sec:existence}) holds and impose further assumptions as needed.
We will present the arguments in a sequence
of steps in the following subsections.

\subsection{A first order condition for the optimum}
\label{sec:foc}

In the following we derive 
a first order condition satisfied by the 
solution $\alpha_{opt}$, which, in view of convexity, will turn out to be a necessary and sufficient condition for an
$\alpha\in {\cal{M}}(T,\bar{\eta})
$ to be a solution of Problem \ref{prob:main}. 

Assume $\alpha,\tilde{\alpha}\in {\cal{M}}(T,\bar{\eta})$.
Define, for $\epsilon\in [0,1]$,
$$\alpha_{\epsilon}(t)=
(1-\epsilon) \alpha(t)+\epsilon \tilde{\alpha}(t).$$
Note that $\alpha_{\epsilon}\in {\cal{M}}(T,\bar{\eta})$ for all $\epsilon$, hence, defining
$$\phi_{\alpha,\tilde{\alpha}}(\epsilon)=\Phi[\alpha_{\epsilon}],$$
we have that if $\alpha=\alpha_{opt}$ then, for any choice of $\tilde{\alpha}\in {\cal{M}}(T,\bar{\eta})$,
$\phi_{\alpha,\tilde{\alpha}}(\epsilon)\geq \phi_{\alpha,\tilde{\alpha}}(0)$ for $\epsilon \in [0,1]$, so that 
\begin{equation}\label{eq:pgz}\phi_{\alpha,\tilde{\alpha}}'(0)\geq 0\qquad\forall \tilde{\alpha}\in {\cal{M}}(T,\bar{\eta}).
\end{equation}
Conversely, if \eqref{eq:pgz} holds, then, since, by Lemma
\ref{lemma:convexity},
$\phi_{\alpha,\tilde{\alpha}}(\epsilon)$ is a strictly convex function, it follows that $\phi_{\alpha,\tilde{\alpha}}(\epsilon)$ is increasing in $\epsilon\in [0,1]$, implying
$\Phi[\tilde{\alpha}]>\Phi[\alpha]$, and since this is valid for all $\tilde{\alpha}\in {\cal{M}}(T,\bar{\eta})$, we conclude that $\alpha=\alpha_{opt}$.
We have thus shown that $\alpha=\alpha_{opt}$ if and only if \eqref{eq:pgz} holds.

We will now explicitly express condition \eqref{eq:pgz}.
Defining
$$S_{\epsilon}(t)={\cal{S}}[\alpha_{\epsilon}],$$
we have, using \eqref{eq:des},
$$\frac{d}{d\epsilon}S_\epsilon(t)\Big|_{\epsilon=0}=- e^{- \alpha(t)-\delta t}v(t)\int_0^T  K(t,r)c(r)e^{\alpha(r)+\delta r}dr$$
$$+e^{-\alpha(t)-\delta t }\int_0^T K(t,r)  c(r)e^{ \alpha(r)+\delta r}v(r)dr,$$
where
$$v(t)=\tilde{\alpha}(t)-\alpha(t).$$
Therefore, using an interchange of the order of integration,
\begin{equation}\label{eq:dp0}\begin{split}
\phi_{\alpha,\tilde{\alpha}}'(0)=&\frac{d}{d\epsilon}\left[\int_0^T w(t)S_{\epsilon}(t)dt\right]\Big|_{\epsilon=0}=\int_0^T w(t)\frac{d}{d\epsilon}S_\epsilon(t)\Big|_{\epsilon=0} dt\\
=&-\int_0^T w(t)e^{- \alpha(t)-\delta t }v(t)\int_0^T K(t,r)  c(r)e^{\alpha(r)+\delta r}dr dt\\
&+\int_0^Tw(t)e^{-\alpha(t)-\delta t }\int_0^T  K(t,r)c(r)e^{ \alpha(r)+\delta r}v(r)drdt\\
=&-\int_0^T w(t)e^{- \alpha(t)-\delta t }v(t)\int_0^T K(t,r)  c(r)e^{\alpha(r)+\delta r}dr dt\\
&+\int_0^T c(t)e^{ \alpha(t)+\delta t}v(t)\int_0^T K(r,t)
w(r)e^{-\alpha(r)-\delta r }
drdt\\
=&-\int_0^T h[\alpha](t)\cdot v(t)dt=\int_0^T h[\alpha](t)\cdot (\alpha(t)-\tilde{\alpha}(t))dt,
\end{split}\end{equation}
where $h[\alpha](t)$ is defined by \begin{equation}\begin{split}\label{eq:h}h[\alpha](t)=&
w(t)e^{- \alpha(t)-\delta t }\int_0^T K(t,r)  c(r)e^{\alpha(r)+\delta r}dr
\\&-c(t)e^{ \alpha(t)+\delta t}\int_0^T K(r,t)
w(r)e^{-\alpha(r)-\delta r }
dr\\
=&
\frac{w(t)e^{- \alpha(t)-\delta t }}{1-e^{-[\bar{\eta}+\delta]T}}\left[\int_0^t  c(r)e^{\alpha(r)+\delta r}dr+e^{-[\bar{\eta}+\delta ]T}\int_t^T c(r)e^{ \alpha(r)+\delta r}dr\right]\\
	&-\frac{c(t)e^{ \alpha(t)+\delta t}}{{1-e^{-[\bar{\eta}+\delta]T}}}\left[e^{-[\bar{\eta}+\delta ]T}\int_0^t  w(r)e^{-\alpha(r)-\delta r }dr+\int_t^Tw(r) e^{-\alpha(r)-\delta r } dr\right].\end{split}\end{equation}
In view of \eqref{eq:pgz},\eqref{eq:dp0}, we have that $\alpha=\alpha_{opt}$ if and only if
$$\int_0^T h[\alpha](t)\cdot (\alpha(t)-\tilde{\alpha}(t))dt\geq 0,\qquad\forall\tilde{\alpha}\in {\cal{M}}(T,\bar{\eta}).$$

We have therefore obtained the following
\begin{lemma}\label{lemma:pre1}
For $\alpha \in {\cal{M}}(T,\bar{\eta})$, define the function $h[\alpha](t)$ by \eqref{eq:h}.
		Then, we have $\alpha=\alpha_{opt}$ if and only if 
	\begin{equation}\label{eq:ili}\int_0^T h[\alpha](t)\tilde{\alpha}(t)dt\leq \int_0^T h[\alpha](t)\alpha(t)dt,\qquad \forall\tilde{\alpha}\in {\cal{M}}(T,\bar{\eta}).\end{equation}
\end{lemma}

For $\alpha\in {\cal{M}}(T)$ we define
\begin{equation}\label{eq:psi}\psi[\alpha](t)=-\int_0^t h[\alpha](s) ds.\end{equation}
Note that $\psi[\alpha]$ is an absolutely continuous function.
A direct calculation, using \eqref{eq:h}, gives
\begin{equation}
    \begin{split}
        \label{eq:psie}\psi[\alpha](t)=&\frac{1}{1-e^{-[\bar{\eta}+\delta]T}}\int_0^t c(r)e^{\alpha(r)+\delta r}dr\int_t^T w(r)e^{- \alpha(r)-\delta r } dr\\&-\frac{1}{e^{[\bar{\eta}+\delta]T}-1}\int_0^t w(r)e^{- \alpha(r)-\delta r }dr\int_t^T c(r)e^{ \alpha(r)+\delta r}dr,
    \end{split}
\end{equation}
which in particular implies
\begin{lemma}\label{lemma:hprop}
Let $\alpha\in {\cal{M}}(T)$. Then the function $h[\alpha]$ is $T$-periodic and satisfies
$$\int_0^T h[\alpha](t)dt=-\psi[\alpha](T)=0,$$
which implies $\psi[\alpha]$ is also $T$-periodic.
\end{lemma}

Our first-order condition \eqref{eq:ili} is re-expressed in a useful way in the following

\begin{lemma}\label{lemma:bh1}
For $\alpha\in {\cal{M}}(T,\bar{\eta})$,
let 
$$\psi_{max}[\alpha]\doteq \max_{t\in \Real }\psi[\alpha](t)=\max_{t\in [0,T)}\psi[\alpha](t).$$

We have $\alpha=\alpha_{opt}$
if and only if 
\begin{equation}\label{eq:vm}t\in \Omega(\alpha)\quad\Rightarrow \quad \psi[\alpha](t)=\psi_{max}[\alpha].
\end{equation}
\end{lemma}

\begin{proof}
Using Lemma
\ref{lemma:hprop}, we have that the function $\psi[\alpha]$ is 
$T$-periodic, and also that
$$\int_{s}^T h[\alpha](t)dt=\psi[\alpha](s).$$
We then have
\begin{equation*}\begin{split}
        &\int_0^T h[\alpha](t)\tilde{\alpha}(t)dt\\&=\int_0^T h[\alpha](t)\int_0^t 1 \,d\tilde{\alpha}(s) dt=\int_0^T \left(\int_s^T h[\alpha](t) dt\right) d\tilde{\alpha}(s)
=\int_0^T \psi[\alpha](s)d\tilde{\alpha}(s),\end{split}
\end{equation*}
so that \eqref{eq:ili} can be written as
$$\int_0^T \psi[\alpha](t)d\tilde{\alpha}(t)\leq \int_0^T \psi[\alpha](t)d\alpha(t)\qquad\forall \tilde{\alpha}\in{\cal{M}}(T,\bar{\eta}),$$
that is
\begin{equation}\label{eq:ili1}
\sup_{\tilde{\alpha}\in{\cal{M}}(T,\bar{\eta})}\int_0^T \psi[\alpha](t)d\tilde{\alpha}(t)= \int_0^T \psi[\alpha](t)d\alpha(t).
\end{equation}
Note now that since $\psi[\alpha](t)$ is continuous and $T-$periodic, it reaches its maximum $\psi_{max}[\alpha]$. We denote the subset of $[0,T)$ on which this maximum is attained by
$$V_{max}[\alpha]=\{t\in [0,T)\;|\; \psi[\alpha](t)=\psi_{max}[\alpha]\}.$$
We have
$$\int_0^T \psi[\alpha](t) d\tilde{\alpha}(t)\leq \psi_{max}[\alpha]\int_0^T 1 d\tilde{\alpha}(t)=\psi_{max}[\alpha]\cdot \tilde{\alpha}(T)=\psi_{max}[\alpha]\cdot T\bar{\eta},$$
and equality is achieved if and only if 
\begin{equation}\label{eq:ccc}
\Omega(\alpha_{opt})\cap [0,T)\subset V_{max}[\alpha].
\end{equation}
Therefore
$$\sup_{\tilde{\alpha}\in{\cal{M}}(T,\bar{\eta})}\int_0^T \psi[\alpha](s)d\tilde{\alpha}(s)=\psi_{max}[\alpha]\cdot T\bar{\eta},$$
so that \eqref{eq:ili1} holds if 
and only if $\alpha$ satisfies
$$\int_0^T \psi[\alpha](s)d\alpha(s)=\psi_{max}[\alpha]\cdot T\bar{\eta},$$
which holds if and only if
\eqref{eq:ccc} holds. We conclude that 
$\alpha$ satisfies  \eqref{eq:ili1}, which is equivalent to \eqref{eq:ili}, if and only if it satisfies
\eqref{eq:ccc}. Since, by Lemma \ref{lemma:pre1},
$\alpha$ satisfies \eqref{eq:ili}
if and only if $\alpha=\alpha_{opt}$, 
we have that 
\eqref{eq:ccc} holds iff 
$\alpha=\alpha_{opt}$.
Since $\psi[\alpha]$ is $T$-periodic, 
\eqref{eq:ccc} is equivalent to \eqref{eq:vm}.
We thus have the desired result.
\end{proof}

\subsection{Regularity of the optimizer: continuity}
\label{sec:continuity}

While the solution $\alpha_{opt}$ of Problem \ref{prob:main}, whose existence is ensured by Theorem
\ref{th:existence}, is {\it{a-priori}} an arbitrary function $\alpha \in {\cal{M}}(T,\bar{\eta})$, 
the arguments in the following will prove that, under some regularity assumptions on the data $c(t),w(t)$,  the solution $\alpha_{opt}$ is more regular.

We use the notation
$$f(t+)=\lim_{s\rightarrow t+}f(s),\quad f(t-)=\lim_{s\rightarrow t-}f(s),$$
for right and left-hand limits. 


\begin{lemma}\label{lemma:htp}
Assume that $t\in \Omega(\alpha_{opt})$, and that the
limits $w(t\pm),c(t\pm)$ exist.
Then we have
\begin{equation}\label{eq:pdh}h[\alpha_{opt}](t+)\geq 0,\quad h[\alpha_{opt}](t-)\leq 0.\end{equation}
\end{lemma}

\begin{proof}
We first note that, for any $\alpha\in {\cal{M}}(T,\bar{\eta})$, since $\alpha$ is monotone, its left- and right-hand limits
$\alpha(t\pm)$
exist at every point. By the definition \eqref{eq:h}, if 
$c(t\pm),w(t\pm)$ exist then
$h[\alpha](t\pm)$ exist, with
\begin{equation}
    \begin{split}
        \label{eq:lim}h[\alpha](t\pm)=&w(t\pm)e^{- \alpha(t\pm)-\delta t }\int_0^T K(t,r)  c(r)e^{\alpha(r)+\delta r}dr\\&-c(t\pm)e^{ \alpha(t\pm)+\delta t}\int_0^T K(r,t)
w(r)e^{-\alpha(r)-\delta r }
dr.
    \end{split}
\end{equation}
This implies that 
the function $\psi[\alpha](t)$, given by \eqref{eq:psi}, has 
right and left-hand derivatives, given by
\begin{equation}\label{eq:ph}\psi[\alpha]'(t\pm)=-h[\alpha](t\pm).\end{equation}
Assume now that $\alpha=\alpha_{opt}$,
which, by Lemma \ref{lemma:bh1}, implies that \eqref{eq:vm} holds.
Thus, assuming $t\in \Omega(\alpha_{opt})$,  $\psi[\alpha_{opt}]$ is maximized at
$t$, so we have
\begin{equation}\label{eq:wa}\psi[\alpha_{opt}]'(t+)\leq 0,\quad\psi[\alpha_{opt}]'(t-)\geq 0\quad\Rightarrow\quad h[\alpha_{opt}](t+)\geq 0,\quad h[\alpha_{opt}](t-)\leq 0.\end{equation}
\end{proof}

\begin{lemma}\label{lemma:ht0}
	Assume $t\in \Omega(\alpha_{opt})$, $w$,$c$ are continuous at $t$, and $c(t)>0,w(t)>0$. Then,
	
(i) $\alpha_{opt}$ is continuous at $t$.
	
(ii) $h[\alpha_{opt}](t)=0$.

(iii) We have
\begin{equation}\label{eq:reg}\alpha_{opt}(t)=\rho(t)\doteq \frac{1}{2}\ln\left(\frac{w(t)}{c(t)}\cdot \frac{\int_0^T K(t,r)  c(r)e^{\alpha_{opt}(r)+\delta r}dr}{\int_0^T K(r,t)w(r)e^{-\alpha_{opt}(r)-\delta r }dr}\right)-\delta t.\end{equation}

(iv) $S_{opt}$ is continuous at $t$.
\end{lemma}

It should be noted that \eqref{eq:reg}, though it will play an important role in out subsequent arguments, is {\it{not}} an explicit 
expression for $\alpha_{opt}(t)$,
since it
depends on $\alpha_{opt}$. 
An explicit expression for 
$\alpha_{opt}^\prime(t)$ is derived 
in Lemma \ref{lemma:am} below, but the arguments leading to this explicit expression depend on the
regularity results that 
are proved in this and the following Subsections.

\begin{proof}
Using the continuity of $c,w$ at $t$, we have, from \eqref{eq:pdh} and \eqref{eq:lim},
\begin{equation}\label{eq:ineq00}e^{ 2\alpha_{opt}(t+)}\leq e^{-2\delta t}\cdot  \frac{w(t)}{c(t)}\cdot \frac{\int_0^T K(t,r) c(r) e^{\alpha_{opt}(r)+\delta r}dr}{\int_0^T K(r,t)w(r) e^{-\alpha_{opt}(r)-\delta r } dr},\end{equation}
\begin{equation}\label{eq:ineq01}e^{ 2\alpha_{opt}(t-)}\geq e^{-2\delta t}\cdot  \frac{w(t)}{c(t)}\cdot \frac{\int_0^T K(t,r) c(r) e^{\alpha_{opt}(r)+\delta r}dr}{\int_0^T K(r,t)w(r) e^{-\alpha_{opt}(r)-\delta r } dr},\end{equation}
which imply
$$
e^{2\alpha_{opt}(t+)}\leq e^{2\alpha_{opt}(t-)}\quad\Rightarrow\quad
\alpha_{opt}(t+)\leq \alpha_{opt}(t-).
$$
On the other hand, since $\alpha_{opt}$ is monotone non-decreasing we must have $\alpha_{opt}(t+)\geq \alpha_{opt}(t-)$, so we conclude that
\begin{equation}\label{eq:ac}\alpha_{opt}(t+)= \alpha_{opt}(t-).
\end{equation}
We have therefore shown (i).

Moreover, \eqref{eq:lim}, the continuity of $c, w$ at $t$, and \eqref{eq:ac} imply that $h[\alpha_{opt}](t+)=h[\alpha_{opt}](t-)$. Together with \eqref{eq:wa}, we conclude that
(ii) holds.

To prove (iii), note that \eqref{eq:reg}
is another way to write the equality $h[\alpha](t)=0$.

To prove (iv), recall that
\begin{equation}\label{eq:sopt0}S_{opt}(t)=\frac{e^{-\alpha_{opt}(t)-\delta t }}{1-e^{-[\bar{\eta}+\delta ]T}}\left[\int_0^t c(r) e^{\alpha_{opt}(r)+\delta r}dr+e^{-[\bar{\eta}+\delta ]T}\int_t^T c(r)e^{\alpha_{opt}(r)+\delta r}dr\right].\end{equation}
The integral terms in \eqref{eq:sopt0}, are always continuous, and the continuity of 
$e^{-\alpha_{opt}(t)}$ follows from the
continuity of $\alpha_{opt}$ at $t$.
\end{proof}

As a consequence of the previous lemma,
we obtain:

\begin{lemma}\label{lemma:perfect} 
Assume $w,c$ are continuous and positive on an open set $U\subset \Real$.
Then,

(i) $\alpha_{opt}$ and 
$S_{opt}$ are continuous at all points $t\in U$.
	
(ii) The set $U\cap \Omega(\alpha_{opt})$ 
contains no isolated points.
\end{lemma}

\begin{proof}
By Lemma \ref{lemma:ht0} we have that
$\alpha_{opt}$ is continuous at every
$t\in U\cap \Omega(\alpha_{opt})$.
If $t\not\in \Omega(\alpha_{opt})$, $\alpha_{opt}$ is constant in a neighborhood of $t$, so it is trivially continuous. We have therefore shown that $\alpha_{opt}$ is continuous at all $t\in U$. 
From \eqref{eq:sopt0} it then follows that $S_{opt}$ is continuous on $U$.

To prove (ii), we note that an
isolated point $t\in U\cap \Omega(\alpha_{opt})$ would
satisfy $\alpha_{opt}(t+)>\alpha_{opt}(t-)$, contradicting the continuity; hence $U\cap \Omega(\alpha_{opt})$ contains no isolated points.

\end{proof}

\subsection{Regularity of the optimizer: differentiability}
\label{sec:differentiability}

We will denote the one-sided derivatives of a function $f$ 
at $t$ (if they exist) by
$$f'(t+)=\lim_{s\rightarrow t+}\frac{f(s)-f(t+)}{s-t},\quad f'(t-)=\lim_{s\rightarrow t-}\frac{f(s)-f(t-)}{s-t}.$$

\begin{lemma}\label{lemma:diff}
Assume that $t\in \Omega(\alpha_{opt})$ is a point at which $w,c$
are differentiable and $w(t)>0,c(t)>0$.

(i) If 
\begin{equation}\label{eq:cc1}
\exists \delta_0>0:\quad w,c \quad{\mbox{are continuous on }}\quad[t,t+\delta_0),
\end{equation}
and if $t$ is a limit point from the right of the set $\Omega(\alpha_{opt})$, that is
\begin{equation}\label{eq:rd1}t\in \overline{\Omega(\alpha_{opt})\cap (t,\infty)},
\end{equation}
then the right-hand derivative $\alpha_{opt}^\prime(t+)$ exists and is finite.

(ii) If the conditions
\begin{equation}\label{eq:cc2}
	\exists \delta_0>0:\quad w,c \quad{\mbox{are continuous on }}\quad(t-\delta_0,t],
\end{equation}
\begin{equation}\label{eq:rd2}t\in \overline{\Omega(\alpha_{opt})\cap (-\infty,t)}
\end{equation}
hold, then the left-hand derivative $\alpha_{opt}^\prime(t-)$ exists and is finite.

(iii) If \eqref{eq:cc1}, \eqref{eq:rd1}, \eqref{eq:cc2} and
\eqref{eq:rd2} hold, then 
$\alpha_{opt}^\prime(t-)=\alpha_{opt}^\prime(t+)$, so that $\alpha_{opt}$ is 
differentiable at $t$.
\end{lemma}

\begin{proof}

We will now prove (i) -- the proof of (ii) is identical.
We thus assume that $t\in \Omega(\alpha_{opt})$, that $c,w$
are differentiable at $t$, and that
\eqref{eq:cc1},\eqref{eq:rd1} hold. 

By \eqref{eq:cc1} and by~\eqref{eq:reg}, see Lemma \ref{lemma:ht0}(iii), we have
\begin{equation}\label{eq:ko}s\in \Omega(\alpha_{opt})\cap [t,t+\delta_0)\quad\Rightarrow\quad\alpha_{opt}(s)=\rho(s),\end{equation}
where
\begin{equation}\label{eq:rho1}
\begin{split}\rho(s)=&\frac{1}{2}\ln\left(\frac{w(s)}{c(s)}\right)\\&+\frac{1}{2}\ln\left( \frac{\int_0^s c(r) e^{\alpha_{opt}(r)+\delta r}dr+e^{-[\bar{\eta}+\delta ]T}\int_s^T c(r) e^{\alpha_{opt}(r)+\delta r}dr}{\int_s^T  w(r)e^{-\alpha_{opt}(r)-\delta r }dr+e^{-[\bar{\eta}+\delta ]T}\int_0^s w(r)e^{-\alpha_{opt}(r)-\delta r }dr}\right)-\delta s.
\end{split}
\end{equation}
We need to prove that the right-hand derivative $\alpha_{opt}^\prime(t+)$ exists.
Indeed we will prove that
\begin{equation}\label{eq:sh}
\alpha_{opt}^\prime(t+)=\rho'(t).
\end{equation}
Note that the assumption that
$w,c$ are differentiable at $t$
implies that the first term in 
\eqref{eq:rho1} is differentiable at $t$,
 and 
the fact that, by Lemma \ref{lemma:ht0}(i),
 $\alpha_{opt}$ is continuous at $t$, implies that the second term is differentiable at $t$. Hence $\rho$ 
is differentiable at $t$.
Thus, fixing $\epsilon>0$, we can find
$\delta_1\in (0,\delta_0)$ (where $\delta_0$ is the constant in \eqref{eq:cc1})  such that
\begin{equation}\label{eq:dc}s\in (t,t+\delta_1)\quad\Rightarrow\quad 
\left|\frac{\rho(s)-\rho(t)}{s-t}-\rho'(t) \right|<\epsilon.\end{equation}
By \eqref{eq:rd1} we can choose $\delta_2\in (0,\delta_1)$ so that $t+\delta_2\in \Omega(\alpha_{opt})$.

We will now show that
\begin{equation}\label{eq:dcon}s\in (t,t+\delta_2]\quad\Rightarrow\quad \left|\frac{\alpha_{opt}(s)-\alpha_{opt}(t)}{s-t}-\rho'(t) \right|<\epsilon.\end{equation}
We consider two cases:

(1) If $s\in (t,t+\delta_2]\cap \Omega(\alpha_{opt})$ then, by
 \eqref{eq:ko},\eqref{eq:dc},
\begin{equation}\label{eq:der}
\left|\frac{\alpha_{opt}(s)-\alpha_{opt}(t)}{s-t}-\rho'(t) \right|=\left|\frac{\rho(s)-\rho(t)}{s-t}-\rho'(t) \right|<\epsilon.
\end{equation}

(2) Assume that $s\in (t,t+\delta_2)$ and 
$s\not\in \Omega(\alpha_{opt})$. Let
$(s_1,s_2)$ be the largest open interval containing $s$ ($s\in (s_1,s_2)$) and contained in the open set $\Real\setminus \Omega(\alpha_{opt})$.
This implies $s_1,s_2\in \Omega(\alpha_{opt})$.
Since $t+\delta_2\in \Omega(\alpha_{opt})$, we must have $s_2\leq t+\delta_2$. 
Since \eqref{eq:rd1} implies 
that the interval $(t,s)$ intersects $\Omega(\alpha_{opt})$, we must have $s_1>t$.
We thus have
$s_1,s_2\in (t,t+\delta_2]$, which implies that 
$\alpha_{opt}$ is continuous at $s_1,s_2$.
Since $(s_1,s_2)\subset \Real\setminus \Omega(\alpha_{opt})$, 
$\alpha_{opt}$ is constant on $(s_1,s_2)$. 
Therefore
$$\alpha_{opt}(s_1)=\alpha_{opt}(s)=\alpha_{opt}(s_2)\qquad\forall s\in (s_1,s_2).$$
We, therefore, conclude that
\begin{equation}\label{eq:lc1}\frac{\alpha_{opt}(s)-\alpha_{opt}(t)}{s-t}\leq
\frac{\alpha_{opt}(s)-\alpha_{opt}(t)}{s_1-t}=\frac{\alpha_{opt}(s_1)-\alpha_{opt}(t)}{s_1-t},
\end{equation}
\begin{equation}\label{eq:lc2}
\frac{\alpha_{opt}(s)-\alpha_{opt}(t)}{s-t}\geq 
\frac{\alpha_{opt}(s)-\alpha_{opt}(t)}{s_2-t}=\frac{\alpha_{opt}(s_2)-\alpha_{opt}(t)}{s_2-t}.
\end{equation}
Since $s_1,s_2\in (t,t+\delta_2]\cap \Omega(\alpha_{opt})$, \eqref{eq:der} implies
\begin{equation}\label{eq:dd1}\frac{\alpha_{opt}(s_1)-\alpha_{opt}(t)}{s_1-t}<\rho'(t)+\epsilon,\end{equation}
\begin{equation}\label{eq:dd2}\frac{\alpha_{opt}(s_2)-\alpha_{opt}(t)}{s_2-t}>\rho'(t)-\epsilon.
\end{equation}
Using \eqref{eq:lc1} and \eqref{eq:dd1} we have
$$\frac{\alpha_{opt}(s)-\alpha_{opt}(t)}{s-t}< \rho'(t)+\epsilon,$$
and using \eqref{eq:lc2} and \eqref{eq:dd2} we have
$$\frac{\alpha_{opt}(s)-\alpha_{opt}(t)}{s-t}>\rho'(t)-\epsilon.$$
We have, thus, shown that \eqref{eq:dcon} holds,
and since this is true for any $\epsilon>0$, we conclude that
$$\alpha_{opt}(t+)=\rho'(t).$$

From (i), and (ii) we have that if all the conditions \eqref{eq:cc1}, \eqref{eq:rd1}, \eqref{eq:cc2} and
\eqref{eq:rd2} hold, then 
$$\alpha_{opt}(t+)=\alpha_{opt}(t-)=
\rho'(t),$$ 
hence $\alpha_{opt}$ is differentiable at $t$.
\end{proof}

\subsection{Regularity of the optimizer: absolute continuity}

\label{sec:absolute}

We now cite a result from advanced real analysis which will be used to prove the absolute continuity of $\alpha_{opt}$ under appropriate conditions (\cite{Jeffery}, Theorem 6.6, \cite{Kannan}, Theorem 7.1.37).

\begin{lemma}\label{lemma:j}
	If a function $F(t)$ is continuous on an 
	interval $[a,b]$, and if $F'(t)$ exists and is finite, except for a countable set of points, and $F'\in L^1[a,b]$, then $F(t)$ is absolutely continuous.
\end{lemma}

We will also use the following basic topological result. The sets $R_+$ ($R_-$) in this result are the sets of points which are {\it{not}} limit points from the right (left) of the set $C$.
\begin{lemma}\label{lemma:top}
Let $C\subset \Real$ be a closed set. Then the sets
$$R_+=\{t\in C\;|\; t\not\in \overline{C \cap (t,\infty)}\},\quad R_-=\{t\in C\;|\; t\not\in \overline{C \cap (-\infty,t)}\}$$
are countable.
\end{lemma}

\begin{proof}
$\Real\setminus C$ is an open set, hence it can be uniquely written as
the union of a countable number of open intervals.
If 
$t\in R_+$ or $t\in R_-$ then it must be an end point of one of these intervals, and since the number of such endpoints is countable, the result follows.
\end{proof}

Using Lemmas \ref{lemma:diff},\ref{lemma:j},\ref{lemma:top} we obtain

\begin{lemma}\label{lemma:diff2}
	Let $U$ be an open set, and
	assume $c,w$ are continuous and positive on $U$, and differentiable at every point of $U$ except for a countable set.
	 Then:
	 
	 (i) The derivative $\alpha_{opt}^\prime(t)$ exists
	at all points $t\in U$ except for a countable set, and we have, at points of differentiability,
\begin{equation}\label{eq:lu}\alpha_{opt}^\prime(t)=\begin{cases}
		\rho'(t) & t\in U\cap \Omega(\alpha_{opt})\\
		0 & t\in U,\;t\not\in  \Omega(\alpha_{opt})
	\end{cases}
\end{equation}
where $\rho$ is given by \eqref{eq:reg}.

(ii) Assuming in addition that
$c,w$ are absolutely continuous on an interval $[a,b]\subset U$, we have that
 $\alpha_{opt}$ is absolutely continuous on $[a,b]$.
\end{lemma}

\begin{proof}
If $t\not\in \Omega(\alpha_{opt})$
then $\alpha_{opt}(t)$ is constant 
in a neighborhood of $t$, hence 
$\alpha_{opt}$ is differentiable at $t$, with $\alpha_{opt}^\prime(t)=0$. We thus need to consider $t\in U\cap \Omega(\alpha_{opt})$.

Since $c,w$ are  continuous on $U$, \eqref{eq:cc1},\eqref{eq:cc2} hold for all $t\in U$. We claim 
that \eqref{eq:rd1},\eqref{eq:rd2} hold for 
all $t\in \Omega(\alpha_{opt})$ except for a countable set. 
Indeed, let $R_+$ be the set of points $t\in
\Omega(\alpha_{opt})$ for which 
\eqref{eq:rd1} does 
{\it{not}} hold. By Lemma \ref{lemma:top} (with $C=\Omega(\alpha_{opt})$), $R_+$ is countable.
In precisely the same way, the set $R_-$ of points $t\in \Omega(\alpha_{opt})$ for which \eqref{eq:rd2} does not hold is countable. 

We thus have that all assumptions of Lemma \ref{lemma:diff}(iii) hold at all
points $t\in U\cap \Omega(\alpha_{opt})$ except for a countable set, implying that $\alpha_{opt}$ is differentiable at all these points, with
$\alpha_{opt}^\prime(t)=\rho'(t)$. We have thus proved (i).

To prove (ii), we 
differentiate \eqref{eq:rho1},
so that we have, for all points $t\in U$ except for a countable set,
\begin{equation}\label{eq:drho}
\begin{split}\rho^\prime(t)=& \frac{1}{2}\left[\frac{w'(t)}{w(t)}-\frac{c'(t)}{c(t)} \right]\\&+\frac{1}{2}\left[\ln\left( \frac{\int_0^t c(r) e^{\alpha_{opt}(r)+\delta r}dr+e^{-[\bar{\eta}+\delta ]T}\int_t^T c(r) e^{\alpha_{opt}(r)+\delta r}dr}{\int_t^T  w(r)e^{-\alpha_{opt}(r)-\delta r }dr+e^{-[\bar{\eta}+\delta ]T}\int_0^t w(r)e^{-\alpha_{opt}(r)-\delta r }dr}\right)\right]^\prime-\delta.
\end{split}
\end{equation}
 By the assumption that $c,w$
 are absolutely continuous and positive on $[a,b]$, we have that the first term in \eqref{eq:drho} is in
 $L^1[a,b]$, and the second term is continuous by the continuity of $\alpha_{opt}$ on $[a,b]$ (which follows from Lemma \ref{lemma:perfect}).
Therefore, $\rho'(t)\in L^1[a,b]$, which, by \eqref{eq:lu}, implies $\alpha_{opt}^\prime(t)\in L^{1}[a,b]$. Thus $\alpha_{opt}$ satisfies all assumptions of Lemma \ref{lemma:j}, hence it is absolutely continuous on $[a,b]$.
\end{proof}

Note in particular that Lemma \ref{lemma:diff2}
implies that, under the conditions of Theorem
\ref{th:main}, $\alpha_{opt}$ is absolutely continuous on {\it{every}} closed interval, as claimed in Theorem
\ref{th:main}.

\subsection{Obtaining an explicit expression for the solution}
\label{sec:explicit}

Our next aim is to derive explicit formulas for $S_{opt}(t)$, and then 
for 
 $\alpha_{opt}^\prime(t)$, when 
$\alpha_{opt}$ is absolutely continuous. An important role will be played by the following identity.

\begin{lemma}\label{lemma:kid}
	Let $\alpha \in {\cal{M}}(T,\bar{\eta})$, and let $S={\cal{S}}[\alpha]$. 
	Then, for all $t$:
	\begin{equation}\label{eq:bi}\psi[\alpha](t)+\frac{h[\alpha](t)}{c(t)}S(t)+M =\frac{w(t)}{c(t)}S(t)^2,\end{equation}
	where 
	$$M=\frac{1}{(e^{[\bar{\eta}+\delta] T}-1)(1-e^{-[\bar{\eta}+\delta] T})}\int_0^T  w(r)e^{-\alpha(r)-\delta r }dr\cdot \int_0^T c(r)e^{\alpha(r)+\delta r}dr.$$
\end{lemma}

\begin{proof}
	Using \eqref{eq:psie} we have
	\begin{equation}\begin{split}\label{eq:viv0}\psi[\alpha](t)
		=&\frac{1}{1-e^{-[\bar{\eta}+\delta] T}}\int_0^T w(r)e^{- \alpha(r)-\delta r } dr\cdot \int_0^tc(r)e^{\alpha(r)+\delta r}dr\\&-\frac{1}{e^{[\bar{\eta}+\delta] T}-1}\int_0^T c(r)e^{ \alpha(r)+\delta r}dr\cdot \int_0^t w(r)e^{- \alpha(r)-\delta r }dr\\
		&-\int_0^t w(r)e^{- \alpha(r)-\delta r }dt\cdot \int_0^t c(r)e^{ \alpha(r)+\delta r}dr.
	\end{split}\end{equation}
	Using \eqref{eq:spex0}
	and \eqref{eq:h} we can write
\begin{equation}\label{eq:viv}
\begin{split}
		\frac{h[\alpha](t)}{c(t)}
		=&\frac{w(t)}{c(t)} S(t)\\-&\frac{e^{\alpha(t)+\delta t}}{1-e^{-[\bar{\eta}+\delta ]T}}\left[\int_0^T  w(r)e^{-\alpha(r)-\delta r }dr-(1-e^{-[\bar{\eta}+\delta ]T})\int_0^t w(r)e^{-\alpha(r)-\delta r }dr\right].
        \end{split}
        \end{equation}
Using \eqref{eq:spex0} and \eqref{eq:viv0}, we have
\begin{equation}
\begin{split}
&\frac{e^{\alpha(t)+\delta t}}{1-e^{-[\bar{\eta}+\delta ]T}}\left[\int_0^T  w(r)e^{-\alpha(r)-\delta r }dr-(1-e^{-[\bar{\eta}+\delta ]T})\int_0^t w(r)e^{-\alpha(r)-\delta r }dr\right]S(t)\\
=&\frac{1}{(1-e^{-[\bar{\eta}+\delta ]T})^2}\left[\int_0^T  w(r)e^{-\alpha(r)-\delta r }dr-(1-e^{-[\bar{\eta}+\delta ]T})\int_0^t w(r)e^{-\alpha(r)-\delta r }dr\right]\\\times
&\left[(1-e^{-[\bar{\eta}+\delta ]T})\int_0^t c(r) e^{\alpha(r)+\delta r}dr+e^{-[\bar{\eta}+\delta ]T}\int_0^T c(r)e^{\alpha(r)+\delta r}dr\right]\\
=&\psi[\alpha](t)+\frac{1}{(1-e^{-[\bar{\eta}+\delta ]T})(e^{[\bar{\eta}+\delta ]T}-1)}\int_0^T  w(r)e^{-\alpha(r)-\delta r }dr\cdot \int_0^T c(r)e^{\alpha(r)+\delta r}dr\\=&\psi[\alpha](t)+M,
\end{split}\nonumber
\end{equation}   
which, together with~\eqref{eq:viv} gives~\eqref{eq:bi}.
\end{proof}

As an immediate consequence
of the above identity, we obtain
\begin{lemma}\label{lemma:identity}
	Assume $u\in C(T)$ is absolutely continuous.
	Let $\alpha \in {\cal{M}}(T,\bar{\eta})$, and let $S={\cal{S}}[\alpha]$. 
	Then we have the following equality:
	$$\int_0^T \frac{w(t)}{c(t)}S(t)^2 u'(t)dt=\int_0^T h[\alpha](t)\left[u(t)+S(t)\frac{u'(t)}{c(t)}\right]dt.$$
\end{lemma}

\begin{proof}
	Multiplying \eqref{eq:bi}
	by $u'(t)$, and integrating over $[0,T]$, taking into account the peridicity of $u(t)$, we obtain
	$$\int_0^T \psi[\alpha](t)u'(t)dt+\int_0^T\frac{h[\alpha](t)}{c(t)}S(t)u'(t)dt =\int_0^T \frac{w(t)}{c(t)}S(t)^2u'(t) dt.$$
	
	Performing an integration by parts, using the facts that
	$\psi[\alpha]'=-h[\alpha]$ and 
	that $\psi[\alpha],u$ are absolutely continuous and $T$-periodic leads to
	$$\int_0^T\psi[\alpha](t)u'(t)dt=-\int_0^T\psi'[\alpha](t)u(t)dt=\int_0^T h[\alpha](t)u(t)dt,$$
	which gives the result.
\end{proof}

\begin{lemma}\label{lemma:srep}
	Under Assumption \ref{a:pc},
	There exists a constant 
	$\lambda>0$ such that, for {\it{a.e.}} $t\in \Omega(\alpha_{opt})$,
	\begin{equation}\label{eq:sss1}S_{opt}(t)=\lambda\cdot \sqrt{\frac{c(t)}{w(t)}}.\end{equation}
\end{lemma}	

\begin{proof}
	The result is trivially true if 
	$\Omega(\alpha_{opt})$ is a set of 
	Lebesgue measure $0$, hence we 
	will assume below that 
	$|\Omega(\alpha_{opt})|>0$.
	Let $v_0\in L^1(T)$ be any function 
	satisfying
	\begin{equation}\label{eq:vp1}	t\not \in \Omega(\alpha_{opt})\quad\Rightarrow\quad
		v_0(t)=0,\end{equation}
	\begin{equation}\label{eq:vp2}
		\int_0^T v_0(t)dt=0.
	\end{equation}
	Define
	$$u(t)=\int_0^t v_0(s)ds,$$
	so that, by \eqref{eq:vp2}, $u(t)$ is $T$-periodic, so that $u\in C(T)$, and is absolutely continuous, with
	$u'=v_0$ {\it{a.e.}}.
	Using Lemma \ref{lemma:identity}, we have
	\begin{equation}\label{eq:ts1}
		\int_0^T \frac{w(t)}{c(t)}S_{opt}(t)^2v_0(t)dt=\int_0^T h[\alpha_{opt}](t)\left[u(t)+\frac{v_0(t)}{c(t)}S_{opt}(t) \right]dt.
	\end{equation}
	Using \eqref{eq:vp1}, and 
	\eqref{eq:vm}, we have
	\begin{equation}\begin{split}\label{eq:ts2}
		&\int_0^T h[\alpha_{opt}](t)u(t)dt
		=-\int_0^T \psi[\alpha_{opt}]'(t)u(t)dt
		=\int_0^T \psi[\alpha_{opt}](t)u'(t)dt
		\\&=
		\int_0^T \psi[\alpha_{opt}](t)v_0(t)dt=\int_{\Omega(\alpha_{opt})\cap [0,T]} \psi[\alpha_{opt}](t)v_0(t)dt=\int_{\Omega(\alpha_{opt})\cap[0,T]} \psi_{max}[\alpha_{opt}]v_0(t)dt
		\\&=\psi_{max}[\alpha_{opt}]\int_{\Omega(\alpha_{opt})\cap [0,T]} v_0(t)dt=\psi_{max}[\alpha_{opt}]\int_0^T v_0(t)dt=0.
	\end{split}\end{equation}
	By
	Lemma \ref{lemma:ht0}, 
	$h[\alpha_{opt}](t)=0$ for 
	$t\in \Omega(\alpha_{opt})$ at which $c,w$ are continuous, and 
	together with \eqref{eq:vp1} we conclude that 
	$$h[\alpha_{opt}](t)\cdot v_0(t)=0\quad{\mbox{for {\it{a.e.}}}}\quad t,$$
	hence
	\begin{equation}\label{eq:ts3}\int_0^T h[\alpha_{opt}](t)\frac{v_0(t)}{c(t)}S_{opt}(t)dt=0.\end{equation}
	Combining \eqref{eq:ts1},\eqref{eq:ts2},\eqref{eq:ts3}, we obtain that
	\begin{equation}\label{eq:fin}
		\int_{\Omega(\alpha_{opt})\cap[0,T]} \frac{w(t)}{c(t)}S_{opt}(t)^2v_0(t)dt=\int_0^T \frac{w(t)}{c(t)}S_{opt}(t)^2v_0(t)dt=0
	\end{equation}
	holds for all functions $v_0$ satisfying 
	\eqref{eq:vp1},\eqref{eq:vp2}.
	Defining 
	\begin{equation}    \label{eq:defm}m=\frac{1}{|\Omega(\alpha_{opt})|}\int_{\Omega(\alpha_{opt})\cap[0,T]} \frac{w(t)}{c(t)}S_{opt}(t)^2dt,\end{equation}
	we also have, for all functions $v_0$ satisfying 
	\eqref{eq:vp1},\eqref{eq:vp2},
	\begin{equation}\label{eq:fin1}
    \begin{split}
		&\int_{\Omega(\alpha_{opt})\cap [0,T]} \left[\frac{w(t)}{c(t)}S_{opt}(t)^2-m\right]v_0(t)dt\\
		&=\int_{\Omega(\alpha_{opt})\cap [0,T]} \frac{w(t)}{c(t)}S_{opt}(t)^2v_0(t)dt-m\int_{\Omega(\alpha_{opt})\cap [0,T]}v_0(t)dt=0-0=0.
        \end{split}
	\end{equation}
	Now let $v\in L^1(T)$ be an arbitrary function satisfying 
	\begin{equation}\label{eq:vp3}	t\not \in \Omega(\alpha_{opt})\quad\Rightarrow\quad
		v(t)=0,\end{equation}
	and define
	$$v_0(t)=\begin{cases}v(t)-\frac{1}{|\Omega(\alpha_{opt})|}\int_{\Omega(\alpha_{opt})\cap [0,T]} v(t)dt & t\in \Omega(\alpha_{opt})\\
	0 & t\not\in \Omega(\alpha_{opt})
	\end{cases}
	$$
	so that $v_0$ satisfies \eqref{eq:vp1},\eqref{eq:vp2}. Then, using 
	\eqref{eq:defm} and \eqref{eq:fin1} we have
    \begin{equation*}\begin{split}
		&\int_{\Omega(\alpha_{opt})\cap [0,T]} \left[\frac{w(t)}{c(t)}S_{opt}(t)^2-m\right]v(t)dt\\
		=&\int_{\Omega(\alpha_{opt})\cap [0,T]} \left[\frac{w(t)}{c(t)}S_{opt}(t)^2-m\right]v_0(t)dt\\&- \left(\frac{1}{|\Omega(\alpha_{opt})|}\int_{\Omega(\alpha_{opt})\cap [0,T]} v(t)dt \right)\int_{\Omega(\alpha_{opt})\cap[0,T]} \left[\frac{w(t)}{c(t)}S_{opt}(t)^2-m\right]dt\\
		=&0- \left(\frac{1}{|\Omega(\alpha_{opt})|}\int_{\Omega(\alpha_{opt})\cap [0,T]} v(t)dt \right)\left[\int_{\Omega(\alpha_{opt})\cap [0,T]} \frac{w(t)}{c(t)}S_{opt}(t)^2-|\Omega(\alpha_{opt})|\cdot m\right]=0.
	\end{split}\end{equation*}
	Since this holds for all $v\in L^1(T)$ 
	satisfying \eqref{eq:vp3}, we conclude that
	$$\frac{w(t)}{c(t)}S_{opt}(t)^2-m=0\qquad{\mbox{for {\it{a.e.}}  }}t\in \Omega(\alpha_{opt}).$$
	Thus, setting $\lambda=\sqrt{m}$, we have \eqref{eq:sss1} for {\it{a.e.}} $t\in \Omega(\alpha_{opt})$.
\end{proof}

\begin{lemma}\label{lemma:am}
	Under Assumption \ref{a:pc}, $\alpha_{opt}$ is differentiable at all points of $U$ apart from a countable set, is 
	absolutely continuous on every closed sub-interval of $U$ and, for {\it{a.e.}} $t\in U$,
	\begin{equation}\label{eq:ef}\alpha_{opt}^\prime(t)=\begin{cases}
			\frac{1}{\lambda}\cdot \sqrt{w(t)c(t)} -\frac{1}{2} \left[\ln \frac{c(t)}{w(t)} \right]'-\delta & t\in U\cap \Omega(\alpha_{opt})\\
			0 & t\in U,\;t\not\in \Omega(\alpha_{opt})
		\end{cases},\end{equation}
	where $\lambda$ is the constant from
	Lemma \ref{lemma:srep}.
\end{lemma}

\begin{proof}
	By 
	Lemma \ref{lemma:diff2}
	we conclude that $\alpha_{opt}$ is
	absolutely continuous on every closed subinterval of $U$. By 
	\eqref{eq:sopt0} this implies that $S_{opt}$ is absolutely continuous on every closed subinterval of $U$.
	We differentiate \eqref{eq:sopt0}, and obtain, for {\it{a.e.}} $t\in U$:
	\begin{equation*}
    \begin{split}S_{opt}'(t)=&-\frac{ e^{-\alpha_{opt}(t)-\delta t }}{1-e^{-[\bar{\eta}+\delta] T}}(\alpha_{opt}'(t)+\delta)\\\times&\left[\int_0^t  c(r)e^{\alpha_{opt}(r)+\delta r}dr+e^{-[\bar{\eta}+\delta] T}\int_t^T 
    c(r)e^{\alpha_{opt}(r)+\delta r}dr\right]+c(t)\\
	&=-(\alpha_{opt}'(t)+\delta)S_{opt}(t)+c(t),
    \end{split}
    \end{equation*}
    or
\begin{equation}\label{eq:eta1}\alpha'_{opt}(t)=\frac{c(t)}{S_{opt}(t)} -\frac{S_{opt}'(t)}{S_{opt}(t)}-\delta\qquad{\mbox{for {\it{a.e.}}} }\;t\in U.\end{equation}	
	Substituting \eqref{eq:sss1} into 
	\eqref{eq:eta1}, we obtain, for {\it{a.e.}}
	$t\in U\cap \Omega(\alpha_{opt})$,
	$$\alpha_{opt}^\prime(t)=\frac{1}{\lambda}\cdot \sqrt{w(t)c(t)} -\frac{1}{2} \left[\ln \frac{c(t)}{w(t)} \right]'-\delta,$$
	and for $t\not\in \Omega(\alpha_{opt})$ we have
	$\alpha_{opt}^\prime(t)=0$. We have thus proved \eqref{eq:ef}.
\end{proof}

\begin{proof}[Proof of Theorem \ref{th:main}]
We apply Lemma \ref{lemma:am} with
$U=\Real$. In this case we also have (setting $\Omega=\Omega(\alpha_{opt})$),
\begin{equation*}\begin{split}
\bar{\eta}&=\frac{1}{T}\int_0^T \alpha_{opt}^\prime(t)dt=\frac{1}{T}\int_{\Omega\cap[0,T]}\alpha_{opt}^\prime(t)dt
\\&=\frac{1}{T}\int_{\Omega\cap [0,T]}\left[\frac{1}{\lambda}\cdot \sqrt{w(t)c(t)} -\frac{1}{2} \left[\ln \frac{c(t)}{w(t)} \right]'\right]dt -\delta \frac{|\Omega\cap [0,T]|}{T}\\
&\Rightarrow\quad\lambda=\frac{\frac{1}{T}\int_{\Omega\cap [0,T]} \sqrt{w(t)c(t)}dt}{\bar{\eta}+\delta\frac{|\Omega\cap [0,T]|}{T}+\frac{1}{2T} \int_{\Omega\cap [0,T]}\left[\ln \frac{c(t)}{w(t)}\right]'dt}.
\end{split}\end{equation*}
We therefore have obtained \eqref{eq:etaopt}-\eqref{eq:lambda1}.
\end{proof}

\subsection{The case of a positive optimizer}
\label{sec:complete}

Under Assumption \ref{a:reg}, if it is the case that $\eta_{opt}(t)=\alpha_{opt}^\prime(t)$ is {\it{a.e.}}
positive, so that $\Omega(\alpha_{opt})=\Real$,  we get that $\eta_{opt}(t)=\eta^*(t)$ for {\it{a.e.}} $t$.
In this case, using the absolute continuity and periodicity of $w(t),c(t)$, we have
$$\frac{1}{T}\int_0^T \left[\ln\frac{c(t)}{w(t)}\right]'dt=
\ln\left(\frac{c(T)}{w(T)}\right)-\ln\left(\frac{c(0)}{w(0)}\right)=0,$$
hence \eqref{eq:lambda1} simplifies to
\begin{equation}\label{eq:hlambda}\lambda=\hat{\lambda}\doteq
\frac{1}{\bar{\eta}+\delta}\cdot \frac{1}{T}\int_0^T \sqrt{w(t)c(t)}dt,\end{equation}
and we thus obtain that
$\eta_{opt}(t)=\hat{\eta}(t)$, where $\hat{\eta}$ is given by
\eqref{eq:hateta}
and $S_{opt}(t)=\hat{S}(t)$, where
$\hat{S}$ is given by \eqref{eq:shat}.
Using this expression, we get that the optimal value of Problem \ref{prob:main}
is 
$$
	\Phi_{\min}=\Phi[\alpha_{opt}]=\int_0^T w(t)\hat{S}(t)dt = \frac{1}{\bar{\eta}+\delta}\cdot\frac{1}{T}\cdot \left(\int_0^T  \sqrt{w(t)c(t)} dt\right)^2.
$$

We have thus obtained
\begin{lemma}\label{cor:pos0}
Under Assumption \ref{a:reg},
if $\eta_{opt}(t)=\alpha_{opt}^\prime(t)$ is {\it{a.e.}} positive,
then it is given by 
\begin{equation}\label{eq:eqp}\eta_{opt}(t)= \hat{\eta}(t)\end{equation}
where $\hat{\eta}$ is explicitly given by \eqref{eq:hateta}.
\end{lemma}

An immediate consequence of this result is that a {\it{necessary condition}}
for $\eta_{opt}$ to be {\it{a.e.}} positive is that $\hat{\eta}(t)$ be {\it{a.e.}} positive. We now show that 
{\it{a.e.}} positivity of $\hat{\eta}(t)$ is also a {\it{sufficient}} condition
for the equality \eqref{eq:eqp}.

\begin{lemma}\label{lemma:pos1}
Under Assumption \ref{a:reg}, assume that
$\hat{\eta}(t)$, given by \eqref{eq:hateta}, is {\it{a.e.}} positive. Then \eqref{eq:eqp} holds.
\end{lemma}

\begin{proof}
Let 
$$\hat{\alpha}(t)=\int_0^t \hat{\eta}(s)ds.$$
We want to prove that $\alpha_{opt}=\hat{\alpha}$.
To do this we will use Lemma 
\ref{lemma:bh1}
and show that \eqref{eq:vm} holds for $\alpha=\hat{\alpha}$.
Since, by our assumption that ${\hat{\eta}}$ is {\it{a.e.}} positive, we have $\Omega(\hat{\alpha})=\Real$,
 to show \eqref{eq:vm} holds we need to show that
$\psi[\hat{\alpha}](t)$ is a constant function.
A direct computation gives
\begin{equation}
\begin{split}
\hat{\alpha}(t)&=\int_0^t \hat{\eta}(s)ds\\&=\int_0^t \left[
\frac{\bar{\eta}+\delta}{ \frac{1}{T}\int_0^T \sqrt{w(s)c(s)}ds}\cdot \sqrt{w(s)c(s)}+\frac{1}{2}\cdot \left[\ln\left(\frac{w(s)}{c(s)}\right)\right]'-\delta \right]ds\\&=\frac{\bar{\eta}+\delta}{\frac{1}{T}\int_0^T \sqrt{w(s)c(s)}ds}\cdot\int_0^t 
 \sqrt{w(s)c(s)}ds+\frac{1}{2}\ln\left(\frac{c(0)w(t)}{w(0)c(t)}\right)-\delta t\\
 &\Rightarrow \quad e^{\hat{\alpha}(t)+\delta t}=
 \sqrt{\frac{c(0)w(t)}{w(0)c(t)}}e^{\frac{\bar{\eta}+\delta}{\frac{1}{T}\int_0^T \sqrt{w(s)c(s)}ds}\cdot\int_0^t 
 	\sqrt{w(s)c(s)}ds}
  \\&\Rightarrow\,\int_0^t c(r)e^{\hat{\alpha}(r)+\delta r}dr=
\sqrt{\frac{c(0)}{w(0)}}\int_0^t\sqrt{w(r)c(r)}\cdot e^{\frac{\bar{\eta}+\delta}{\frac{1}{T}\int_0^T \sqrt{w(s)c(s)}ds}\cdot\int_0^r 
	\sqrt{w(s)c(s)}ds}dr\\
&=\frac{\frac{1}{T}\int_0^T \sqrt{w(s)c(s)}ds}{\bar{\eta}+\delta}\cdot\sqrt{\frac{c(0)}{w(0)}}\left(e^{\frac{\bar{\eta}+\delta}{\frac{1}{T}\int_0^T \sqrt{w(s)c(s)}ds}\cdot\int_0^t
	\sqrt{w(s)c(s)}ds}-1\right),
    \end{split}
\end{equation}
and similarly we obtain
\begin{equation}\begin{split}&\int_0^t w(r)e^{-\alpha_{opt}(r)-\delta r }dr\\=&\frac{\frac{1}{T}\int_0^T \sqrt{w(s)c(s)}ds}{\bar{\eta}+\delta}\sqrt{\frac{w(0)}{c(0)}}\left(1-e^{-\frac{\bar{\eta}+\delta}{ \frac{1}{T}\int_0^T \sqrt{w(s)c(s)}ds}\cdot\int_0^t 
	\sqrt{w(s)c(s)}ds}\right).
    \end{split}\end{equation}
Substituting these results into \eqref{eq:h}, we obtain
$$h[\hat{\alpha}](t)=0,\quad\forall t,$$
which, by \eqref{eq:psi}, implies that
$$\psi[\hat{\alpha}](t)=0,\quad\forall t,$$
so we have shown that $\psi[\hat{\alpha}]$ is constant, as needed.
\end{proof}

We now note that, by \eqref{eq:hateta}, the condition that $\hat{\eta}$ is everywhere
positive is equivalent to:
$\bar{\eta}>\bar{\eta}_m$, where
$\bar{\eta}_m$ is given by
\eqref{eq:etam}.
The results of 
Lemma \ref{cor:pos0} and Lemma \ref{lemma:pos1} thus imply 
Theorem \ref{th:pos}.

%

\subsection{The limit $\bar\eta\rightarrow 0$}
\label{sec:e0}

We now use the first order condition for the optimizer in order to prove Theorem \ref{prop:eta0} of Subsection \ref{sec:lh}.

\begin{proof}[Proof of Theorem \ref{prop:eta0}]
	Fix $\epsilon>0$.
	Assume by way of contradiction that $\{\bar{\eta}_k\}$ is a positive sequence with
	\begin{equation}\label{eq:l0}
	\lim_{k\rightarrow \infty} \bar{\eta}_k=0, 
	\end{equation}
	but with $\Omega(\alpha_{opt,k})\not\subset M_{\epsilon}$, where $\alpha_{opt,k}$ is the solution of Problem \ref{prob:main}
	corresponding to $\bar{\eta}_k$.
	Then for each $k$ we can choose $t_k\in [0,T)$ so that $t_k\in \Omega(\alpha_{opt,k})$ but $t_k\not\in M_{\epsilon}$, that is $f(t_k)\leq f_{\max}-\epsilon$. Since the sequence 
	$\{t_k\}$ is bounded, we can assume, by 
	taking a subsequence, that $\{t_k\}$ converges, and denote its limit by $t^*$. By continuity of $f$, we have
	\begin{equation}\label{eq:pl0}
		f(t^*)\leq f_{\max}-\epsilon.
	\end{equation}
Using \eqref{eq:psie} we have
	\begin{equation}\begin{split}\label{eq:pq0}
		(1-e^{-[\bar{\eta}+\delta]T})\psi[\alpha_{opt,k}](t)=&\int_0^{t} c(r)e^{\alpha_{opt,k}(r)+\delta r}dr\int_{t}^T w(r)e^{- \alpha_{opt,k}(r)-\delta r } dr\\-&e^{-[\bar{\eta}_k+\delta ]T}\int_0^{t} w(r)e^{- \alpha_{opt,k}(r)-\delta r }dr\int_{t}^T c(r)e^{ \alpha_{opt,k}(r)+\delta r}dr.\end{split}\end{equation}
	Since
	$$t\in [0,T]\Rightarrow 0\leq \alpha_{opt,k}(t)\leq T\bar{\eta}_k$$
	\eqref{eq:l0} implies 
	$\alpha_{opt,k}(t)\rightarrow 0$ uniformly on $[0,T]$ as $k\rightarrow \infty$. By \eqref{eq:pq0}, this implies that
	\begin{equation}\label{eq:pq3}\lim_{k\rightarrow \infty}(1-e^{-[\bar{\eta}+\delta]T})\psi[\alpha_{opt,k}](t)=f(t)\quad{\mbox{uniformly on}}\; [0,T],
	\end{equation}
    where $f(t)$ is defined by 
    \eqref{eq:deff}.
	which in turn implies that
	\begin{equation}\label{eq:pq2}\lim_{k\rightarrow \infty}(1-e^{-[\bar{\eta}+\delta]T})\psi_{\max}[\alpha_{opt,k}]=f_{max}.\end{equation}
	Since $t_k\in \Omega(\alpha_{opt,k})$, we have,
	by Lemma \ref{lemma:bh1},
	$$\psi[\alpha_{opt,k}](t_k)=\psi_{\max}[\alpha_{opt,k}],$$
	hence, using, \eqref{eq:pq2},
	\begin{equation}\label{eq:pq4}\lim_{k\rightarrow \infty}(1-e^{-[\bar{\eta}+\delta]T})\psi[\alpha_{opt,k}](t_k)=f_{max}.\end{equation}
	Writing,
	\begin{equation}\label{eq:pq5}
		f(t^*)=[f(t^*)-f(t_k)]+[f(t_k)-(1-e^{-[\bar{\eta}+\delta]T})\psi[\alpha_{opt,k}](t_k)]+(1-e^{-[\bar{\eta}+\delta]T})\psi[\alpha_{opt,k}](t_k),\end{equation}
	we have that the first term on the right-hand side goes to $0$ as $k\rightarrow \infty$ by continuity of $f$, the second term goes to $0$ by
	\eqref{eq:pq3}, and the third term goes to $f_{max}$ by \eqref{eq:pq4}.
	Hence, taking the limit $k\rightarrow \infty$ in \eqref{eq:pq5} leads to
	$$f(t^*)=f_{max},$$
	but this contradicts \eqref{eq:pl0}, 
	and this contradiction concludes the proof.
\end{proof}

\subsection{Discontinuities in 
the data}
\label{sec:jumps}

%

Theorem \ref{th:main} has shown that when the functions 
$c,w$ are absolutely continuous the optimizer is an
absolutely continuous measure. 
Theorem \ref{th:discont} examines the 
 the case in which 
$c$ or $w$ has discontinuity points,
and shows that these can lead to atomic components in the optimizing measure.

\begin{proof}[Proof of theorem \ref{th:discont}]

(i) The explicit expression \eqref{eq:ef1}, for $t\in U$,
follows from Lemma~\ref{lemma:am}.

(ii) Assume now that $t\in D_+$, that is
\begin{equation}\label{eq:ju}\frac{w(t+)}{c(t+)}<\frac{w(t-)}{c(t-)}.
\end{equation}
To prove that $t\not\in  \Omega(\alpha_{opt})$, assume by way of contradiction that
$t\in \Omega(\alpha_{opt})$. Then,
by Lemma \ref{lemma:htp} we have
$h[\alpha_{opt}](t+)\geq 0$, $h[\alpha_{opt}](t-)\leq 0$, which, written explicitly, give
\begin{equation*}\label{eq:ineq1}e^{ 2\alpha_{opt}(t+)}\leq e^{-2\delta t}\cdot  \frac{w(t+)}{c(t+)}\cdot \frac{\int_0^T K(t,r) c(r) e^{\alpha_{opt}(r)+\delta r}dr}{\int_0^T K(r,t)w(r)e^{-\alpha_{opt}(r)-\delta r } dr},\end{equation*}
\begin{equation*}\label{eq:ineq2}e^{ 2\alpha_{opt}(t-)}\geq e^{-2\delta t}\cdot \frac{w(t-)}{c(t-)}\cdot \frac{\int_0^T K(t,r) c(r) e^{\alpha_{opt}(r)+\delta r}dr}{\int_0^T K(r,t)w(r)e^{-\alpha_{opt}(r)-\delta r } dr}.\end{equation*}
Together with \eqref{eq:ju}, these imply
$$e^{2\alpha_{opt}(t+)}< e^{2\alpha_{opt}(t-)},$$
so that $\alpha_{opt}(t+)<\alpha_{opt}(t-)$. But this contradicts the fact that $\alpha_{opt}$ is monotone 
non-decreasing. This contradiction proves that 
$t\not\in \Omega(\alpha_{opt})$.

(iii) Assume that $t\in D_-$, that is
	\begin{equation}\label{eq:jd}\frac{w(t+)}{c(t+)}>\frac{w(t-)}{c(t-)},\end{equation} 
and also that the condition \eqref{eq:ts} holds.
This condition implies that there exist sequences $\{t^+_k\},\{t^-_k\}$ 
of points with $t^-_k<t<t^+_k$ and $\lim_{k\rightarrow \infty} t^{\pm}_k=t$, so that $t^{\pm}_k\in \Omega(\alpha_{opt})$ and $c,w$ are continuous at $t^{\pm}_k$. 
Hence, applying Lemma \ref{lemma:ht0}, we conclude that
$$\alpha_{opt}(t^{\pm}_k)= \frac{1}{2}\ln\left(\frac{w(t^{\pm}_k)}{c(t^{\pm}_k)}\cdot \frac{\int_0^T K(t^{\pm}_k,r) c(r) e^{\alpha_{opt}(r)+\delta r}dr}{\int_0^T K(r,t^{\pm}_k)w(r)e^{-\alpha_{opt}(r)-\delta r } dr}\right)-\delta t^{\pm}_k,$$
and taking the limit $k\rightarrow\infty$ 
leads to 
$$\alpha_{opt}(t\pm)= \frac{1}{2}\ln\left(\frac{w(t\pm)}{c(t\pm)}\cdot \frac{\int_0^T K(t,r) c(r) e^{\alpha_{opt}(r)+\delta r}dr}{\int_0^T K(r,t)w(r)e^{-\alpha_{opt}(r)-\delta r } dr}\right)-\delta t.$$
These equalities imply 
$$\alpha_{opt}(t+)-\alpha_{opt}(t-)
=\frac{1}{2}\ln\left(\frac{w(t+)}{c(t+)}\frac{c(t-)}{w(t-)} \right)=
\ln\left(\frac{\varphi(t-)}{\varphi(t+)} \right)>0,$$
so that we have \eqref{eq:js}.

Finally, \eqref{eq:sjump} follows from
\eqref{eq:js}, using \eqref{eq:sopt0}.
\end{proof}

We now prove Proposition \ref{prop:sd}, which gives a sufficient condition under which, for $\bar{\eta}$ sufficiently small, the measure $\mu_{opt}$ is purely atomic.

\begin{proof}[Proof of Proposition \ref{prop:sd}]
By invariance with respect to translations, it suffices to prove the result assuming $t^*=T$, so that we must prove, under the assumption that 
$w(t)\equiv 1$ and $c(t)$ is strictly monotone increasing on $(0,T)$, that $\alpha_{opt}=\alpha_0$, where
$$\alpha_{0}(t)=\begin{cases}
	0 & 0\leq  t<   T\\
	\bar{\eta}T & t=T.
\end{cases}$$
We will show that, under our assumptions, and when 
$\bar{\eta}>0$ is sufficiently small, $\alpha_0$ satisfies the condition \eqref{eq:vm}, so that Lemma \ref{lemma:bh1} will imply $\alpha_{opt}=\alpha_0$.
Since $\Omega(\alpha_0)=\{kT\;|\; k\in \Integer\}$, it suffices to show that 
the $T$-periodic function $\psi[\alpha_0]$ 
given by \eqref{eq:psie} 
attains its maximum
at $t=0$. 
We have
\begin{equation*}
    \begin{split}
        \psi[\alpha_0](t)=&\frac{1}{1-e^{-[\bar{\eta}+\delta]T}}\cdot \frac{1}{\delta}\cdot (e^{-\delta t}-e^{-\delta T})\int_0^t c(r)e^{\delta r}dr\\&-\frac{e^{-[\bar{\eta}+\delta]T}}{1-e^{-[\bar{\eta}+\delta]T}}\cdot  \frac{1}{\delta}\cdot (1-e^{-\delta t})\int_t^T c(r)e^{\delta r}dr,
    \end{split}
\end{equation*}
the condition that $\psi[\alpha_0](t)<\psi[\alpha_0](0)=0$ for 
$t\in (0,T)$, is equivalent to
\begin{equation}\label{eq:cx}\bar{\eta}<\frac{1}{T}\ln\left(\frac{(e^{\delta t}-1)\int_t^T c(r)e^{\delta r}dr}{(e^{\delta T}-e^{\delta t})\int_0^t c(r)e^{\delta r}dr}\right).\end{equation}
We therefore need to show that 
the right-hand side of \eqref{eq:cx} is strictly positive for $t\in (0,T)$, that is 
\begin{equation}\label{eq:cx1}
\frac{1}{e^{\delta T}-e^{\delta t}}\int_t^T c(r)e^{\delta r}dr>\frac{1}{e^{\delta t}-1}\int_0^t c(r)e^{\delta r}dr,
\end{equation}
By the integral mean value theorem, we have 
\begin{equation}\label{eq:cx2}\frac{1}{e^{\delta t}-1}\int_0^t c(r)e^{\delta r}dr=\frac{\int_0^t c(r)e^{\delta r}dr}{\int_0^t e^{\delta r}dr}=c(t_1),
\end{equation}
where $t_1\in (0,t)$, 
and
\begin{equation}\label{eq:cx3}\frac{1}{e^{\delta T}-e^{\delta t}}\int_t^T c(r)e^{\delta r}dr=\frac{\int_t^T c(r)e^{\delta r}dr}{\int_t^T e^{\delta r}dr}=c(t_2),
\end{equation}
where $t_2\in (t,T)$.
By the monotonicity of $c(t)$, we have
$c(t_2)>c(t_1)$, so that \eqref{eq:cx2},\eqref{eq:cx3} imply 
\eqref{eq:cx1}, completing the proof.
\end{proof}

\section{Discussion}
\label{sec:discussion}

This work has been devoted to the study of a 
`simple' optimal control problem, in which 
the aim is to regulate the long-time behavior of
a periodic system to achieve optimal performance, under a constraint on the total effort expended per period. In the context of this problem, it is 
natural to allow the controls to be arbitrary positive measures
$\mu$, where $\mu([t_1,t_2))$ 
is interpreted as the amount of effort expended during the time-interval $[t_1,t_2)$.
This entails that the problem is beyond the scope of the standard optimal control-theory 
framework which assumes that controls are bounded functions.

Throughout the study, we have focused on an example of optimal clearing of a pollutant flowing to a system at a periodic rate.  In Appendix~\ref{app:alternativeInterpretation} we provide an alternative interpretation to the same mathematical problem. 
Another interpretation of the same problem is given in \cite{ali2021maximizing}, involving the maximization of average throughput in a periodic system with a `bottleneck' entrance. An additional application of the same problem involves the question of preventing the spread of an infectious disease with a
seasonally-varying transmission rate, by periodic vaccination - this application will be detailed in a future work. The variety of possible interpretations and applications of our problem shows that it is a quite fundamental one, motivating our quest
to gain a detailed understanding of the
structure of its solution.

Our analyses, as well as the numerical results, have enabled us to understand key features of the optimizing measure which is the solution to our problem.
We have proved that
\begin{itemize}
	\item When the data (the functions $c(t)$,$w(t)$) are absolutely continuous, and differentiable apart from a countable set, the optimizing measure is absolutely continuous, and an explicit formula is given for its density, on its support.
	
	\item Under the above assumptions on $c(t),w(t)$, when the cumulative allowable effort per period ($\bar{\eta}$)
	 is sufficiently large, the support of the optimizing measure is the entire real line, which means that some effort is made at all times. 
	 However, for smaller values of
	 $\bar{\eta}$ there can exist
	 time-intervals which are outside this support, meaning that no effort is made during these periods, and as $\bar{\eta}\rightarrow 0$ the measure of the support goes to $0$.
\end{itemize}

We have focused, in particular, 
on the case in which either $c(t),w(t)$
have points of discontinuity, and have shown that.
\begin{itemize}
	 \item Discontinuity points
	 of either $c(t)$ or $w(t)$ can give rise to atomic ($\delta$-function) components
	 of the optimizing measure. However, this can occur only at those points at which the function  $c(t)/w(t)$ has 
	 a {\it{downward}} jump. In this case the optimizing measure can consist of purely atomic components or of a mixture 
	 of a continuous and atomic components. 
	 
	 \item Points at which $c(t)/w(t)$ has an
	 upward jump are always {\it{outside}} the support of the optimizing measure.
\end{itemize}

Thus, in terms of the interpretation 
of the problem as that of determining an optimal clearing effort for a pollutant, and assuming $w$ is constant, we have that when the pollution inflow rate increases abruptly at some point in time, one should not apply any effort in 
a time interval including this point.
On the other hand when the pollution 
inflow rate decreases abruptly, then it
is sometimes advisable to apply a concentrated effort at this time. 

While our results provide considerable understanding of the structure of the optimizing measure, there are several aspects,
mentioned throughout the text, in which our results are incomplete. In particular, 
while our results provide an explicit expression for the absolutely continuous component of the optimizing measure on its support $\Omega(\alpha_{opt})$, we do not have a complete and explicit characterization of the
set $\Omega(\alpha_{opt})$. Therefore we have also used numerical computation to gain insights.

Some phenomena observed through the numerical computations are not completely explained by the analytical results. Let us mention some observations which still await mathematical proof:
In our numerical computations, we examined the effect of the constraint on the cumulative effort, determined by the value $\bar{\eta}$, on the solution of Problem \ref{prob:main}.
Denoting this solution, in dependence on $\bar{\eta}$, by $\alpha_{opt,\bar{\eta}}$, our numerical results suggest the following monotonicity properties:

\begin{itemize}
\item[(1)] The set $\Omega(\alpha_{opt,\bar{\eta}})$ 
is monotone increasing with respect to 
$\bar{\eta}$: 
$$\bar{\eta}_1<\bar{\eta}_2\quad\Rightarrow\quad\Omega(\alpha_{opt,\bar{\eta}_1})\subset \Omega(\alpha_{opt,\bar{\eta}_2}).$$

\item[(2)] When $\alpha_{opt,\bar{\eta}}$ is absolutely continuous for all $\bar{\eta}$ (e.g., under Assumption \ref{a:reg}), the corresponding density is monotone with respect to $\bar{\eta}$:
$$\bar{\eta}_1<\bar{\eta}_2\quad\Rightarrow\quad\alpha_{opt,\bar{\eta}_1}'(t) \leq \alpha_{opt,\bar{\eta}_2}'(t),\qquad\forall t.$$

\item[(3)] The mass of atomic components of the measure are (weakly)
monotone with respect to $\bar{\eta}$.
If $\bar{\eta}_1<\bar{\eta}_2$ and if 
$\alpha_{opt,\bar{\eta}_1}$ has a 
discontinuity at $t^*$, then so does 
$\alpha_{opt,\bar{\eta}_2}$, and
$$\alpha_{opt,\bar{\eta}_2}(t^*+)-\alpha_{opt,\bar{\eta}_2}(t^*-)\geq \alpha_{opt,\bar{\eta}_1}(t^*+)-\alpha_{opt,\bar{\eta}_1}(t^*-).$$
\end{itemize}

In fact all these observations follow from the following monotonicity conjecture:

\begin{conjecture}\label{conj:1}
	If  $0<\bar{\eta}_1<\bar{\eta}_2$, then
	the measure $\mu_{\alpha_{opt},\bar{\eta}_2}$ dominates the measure $\mu_{\alpha_{opt,\bar{\eta}_1}}$ in the sense that, for any $t<s$,
	$\mu_{\alpha_{opt,\bar{\eta}_1}}([t,s))\leq \mu_{\alpha_{opt,\bar{\eta}_2}}([t,s))$, or, equivalently
	$$\alpha_{opt,\bar{\eta}_1}(s)-\alpha_{opt,\bar{\eta}_1}(t)<\alpha_{opt,\bar{\eta}_2}(s)-\alpha_{opt,\bar{\eta}_2}(t).$$
\end{conjecture}
This means that when the allowed cumulative effort is 
increased, effort made at {\it{any}} point in time can only increase.
While this conjecture is supported by all of our numerical observations, we  have not been able to prove it, and leave its proof as an open question.

Concentrating on a simple model problem has allowed us to employ specific analytical tools which might not easily extend to more general problems - in particular, 
the ability to solve the differential equation \eqref{eq:s0}
in the explicit form \eqref{eq:spex}
has been exploited. It is of great interest to extend our investigation to more elaborate 
problems. This includes studying problems involving optimal regulation of phenomena described by systems of linear periodic differential equations, as well as those involving nonlinear differential equations, with the similar aim of characterizing the behavior of optimizers. The extent to which the ideas employed here can be extended to treating other problems, or the alternative ideas which need to be developed,  should be studied in future work. The investigation of the particular `model problem' performed in this work offers us a window through which we can glimpse some phenomena which should also occur in a more general context, and it will be of interest to learn what additional phenomena appear when the problem is generalized.

\section*{Acknowledgments}
This research was supported by the Israel Science Foundation (grant no. 3730/20) within the KillCorona-Curbing Coronavirus Research Program, and by the Israel Science Foundation (ISF) grant 1596/23.
 \appendix
 \section{An additional interpretation of the optimal control problem}\label{app:alternativeInterpretation}
 The interpretation of the problem introduced in Section \ref{sec:intrducing}, as a problem of optimal clearing 
 of a pollutant flowing into the system at a periodic rate, is only one possible interpretation. Let us propose another setting in which the same optimization problem arises. Assume now that $S(t)$ represents the amount of a resource in the environment. This resource flows into the environment at a given periodically dependent rate $c(t)$ (with period $T$), degrades at a first-order rate $\delta 
 S$, and is extracted by us at the rate $\eta(t)S$, where $\eta(t)$ is our extraction effort, which also depends $T$-periodically on time. Then the amount of resource $S$ satisfies the differential equation \eqref{eq:s0}, and in the long term the solution approaches the periodic solution $S(t)$.
 Thus, in the long term, the amount 
 of resource extracted per period, which we wish to maximize, under the constraint 
 \eqref{eq:ce} on the time-averaged effort,
 is 
 $$\int_0^T \eta(t)S(t)dt.$$
 By integrating \eqref{eq:s0} over 
 a period, and using the periodicity 
 of $S(t)$, we obtain
 $$\int_0^T \eta(t)S(t)dt=\int_0^T c(t)dt-\delta \int_0^T S(t)dt.$$
 Our aim is thus equivalent to 
 minimizing $\int_0^T S(t)dt$, which is identical to the problem introduced in Section \ref{sec:intrducing}, with
 $w(t)\equiv 1$.
 
 \section{Numerical methods}
 \label{sec:numerical}
 We numerically approximate the solution of Problem \ref{prob:main} employing Matlab's {\tt fmincon} function for optimization with constraints
 \begin{lstlisting}[style=Matlab-editor]
 	%% Employ minimization with constraints to find alpha
 	[alpha,Phi]=fmincon(@(x)computePhi(x,delta,t,c,w),alpha0,A,b,[],[],lb,ub,[],optimset('Algorithm','sqp'));
 \end{lstlisting}
 Our approach involves selecting a discrete grid over the interval~$[0,T]$ as follows:
 \[
 0=t_0<t_1<t_2<\cdots<t_K=T
 \]
 with corresponding values~$\alpha_k=\alpha(t_k)$ such that:
 \[
 \alpha_0=0,\quad \alpha_K=T\eta\bar.
 \]
 The grid does not need to be uniform.  By default, we use a uniform grid with some concentration of points around the discontinuities of~$c(t),w(t)$.
 To enforce the monotonicity of~$\alpha$, we apply the linear constraints
 \[
 \alpha_{k-1}\le\alpha_{k},\quad k=1,2,\cdots,K.
 \] We do so by properly setting the parameters of the {\tt fmincon} parameters parameters~${\tt A}$ and~${\tt b}$.
 Our aim is to minimize~$\Phi[\alpha]$ as defined by \eqref{eq:dPhi}.  To do so, we employ Matlab's {\tt cumtrapz} function, which uses the composite trapezoidal rule for non-uniform grids, for the numerical integration required to compute~$\Phi[\alpha]$:
 \begin{lstlisting}[style=Matlab-editor]
 	%% Compute Phi[alpha]
 	function Phi=computePhi(alpha,delta,t,c,w)
 	
 	aux=cumtrapz(t,c.*exp(alpha+delta*t));
 	caux=(aux(end)-aux)*exp(-alpha(end)-delta*t(end));
 	S=exp(-alpha-delta*t).*(aux+caux)/(1-exp(-alpha(end)-delta*t(end)));
 	
 	Phi=trapz(t,w.*S);
 	
 	end
 \end{lstlisting}
 
 Within {\tt fmincon}, we select the option of the optimization algorithm {\tt sqp} because the optimizer may lie on the boundary of the feasible set, i.e., for some~$k$, $\alpha_k=\alpha_{k-1}$.  In this case, the default interior point algorithm of {\tt fmincon} is less effective. The supplementary material includes the full Matlab code for solving the optimization problem and for creating the graphs in this paper.

\end{document}